\documentclass{article}
\usepackage[utf8]{inputenc} 
\usepackage[T1]{fontenc} 
\usepackage{subcaption}
\usepackage{caption}
\usepackage{hyperref}
\hypersetup{
    colorlinks=true,   
    linkcolor=blue,     
    citecolor=blue,     
    urlcolor=blue,     
    pdftitle={},  
    pdfauthor={Dongwoo Gang} 
}
\usepackage{url}       
\usepackage{booktabs}     
\usepackage{amsfonts}    
\usepackage{lipsum}
\usepackage{color}
\usepackage{graphicx}
\usepackage{amsthm}
\usepackage{mathbbol}
\usepackage{bbm}
\usepackage{setspace}
\usepackage{authblk}
\usepackage{amsmath}
\usepackage{algorithm}
\usepackage{algpseudocode}
\usepackage{verbatim}
\usepackage{tikz}
\usepackage{tikz-cd}
\usepackage{tikz-3dplot}
\usepackage{geometry}
\usepackage[font=small,labelfont=bf]{caption}
\usetikzlibrary{calc,arrows.meta}
\newtheorem{theorem}{Theorem}[section]
\newtheorem{lemma}[theorem]{Lemma}
\newtheorem{example}{Example}

\newtheorem{remark}{Remark}[section]

\theoremstyle{definition}
\newtheorem{definition}{Definition}[section]

\title{Oriented Grassmannian Bundle, Normal Curvature Reduction, and Persistent Homology}
\author{Dongwoo Gang}
\date{}
\begin{document}
\maketitle
\onehalfspacing

\begin{abstract}
We consider a smooth closed orientable submanifold $M \subset \mathbb{R}^D$ with narrow cycles. We embed $M$ into a scaled oriented Grassmannian bundle via the Gauss map in order to enlarge the scale of these cycles. Under mild assumptions, we show that this embedding reduces the normal curvature of the embedded submanifold in directions where the original normal curvature is large. For smooth closed hypersurfaces, we further show that this construction increases the distance between antipodal points of narrow cycles for fixed volume.

We then obtain an explicit range of radii for which the ambient Čech complex on this Grassmannian bundle is homotopy equivalent to the embedded manifold, yielding lower bounds on the scales at which the Čech filtration recovers the homology of $M$. Since the distance induced by the embedding depends on both positions and oriented tangent spaces, we work with Whitney $C^1$ convergence of embeddings and prove that the associated Čech persistent homology is stable with respect to the interleaving distance. Finally, we describe a procedure for computing a distance matrix for a finite subset with respect to this embedding and illustrate the construction on several examples, including an approximate quasi-halo orbit in the Saturn--Enceladus system.
\end{abstract}

\section{Introduction}

\emph{Persistent homology} computes barcodes from filtered complexes constructed from metric spaces, including the \emph{\v{C}ech filtration} and the \emph{Vietoris--Rips filtration} \cite{zomorodian2004computing, oudot2015persistence}. If $M$ is a Riemannian manifold, by the nerve lemma, the Čech complex $\check{C}(M;r)$ is homotopy equivalent to $M$ whenever $r$ is smaller than the convexity radius of $M$, while the systole and the filling radius bound the range of $r$ for which the Vietoris--Rips complex $\mathrm{VR}(M;r)$ has the same first and top-dimensional homology groups as $M$ \cite{virk20201, lim2024vietoris, balitskiy2025geometric}.

The \emph{reach} (or \emph{normal injectivity radius}) $\mathsf{rch}_{\mathbb{R}^D}(M)$ of a smooth submanifold $M \subset \mathbb{R}^D$ is the supremum of $\tau>0$ such that every point $x \in \mathbb{R}^D$ with $d_{\mathbb{R}^D}(x,M)<\tau$ has a unique nearest point in $M$ \cite{federer1959curvature, aamari2019estimating}. This agrees with the supremum of radii for which the normal exponential map is injective. The reach of $M$ provides bounds on the scales $r$ where the Vietoris--Rips or Čech complex built from sufficiently dense points of $M$ in the ambient Euclidean metric \cite{niyogi2008finding, kim2019homotopy}, and the \emph{metric thickening} of $M$ \cite{adams2019metric, adams2022metric, adams2024persistent}, the Wasserstein analogue of the Vietoris--Rips complex, are homotopy equivalent to $M$. 

Aamari et al. \cite{aamari2019estimating} showed that the reach of a submanifold $M \subset \mathbb{R}^D$ can be expressed as the minimum of the reciprocal of the \emph{maximal normal curvature} of $M$ and the minimal width of its \emph{bottlenecks}, pairs of points in $M$ whose connecting line segment is orthogonal to both tangent spaces. For example, for a highly eccentric ellipse in $\mathbb{R}^2$ or a thin $2$-torus with small systole in $\mathbb{R}^3$ as in Figure~\ref{ellipse_torus}, the normal exponential map ceases to be injective at a radius comparable to the reach, since the points of large curvature (marked in red) induce focal points and the bottlenecks (marked in blue) lie close to each other in the Euclidean metric. 

In both examples, the indicated bottlenecks have opposite normal directions. This suggests that, for an orientable submanifold $M \subset \mathbb{R}^D$, embedding $M$ into an augmented ambient space, such as the \emph{oriented Grassmannian bundle} $\mathbb{R}^D \times \mathbf{Gr}^+(D,d)$, can enlarge the normal injectivity radius of $M$. In the product metric on $\mathbb{R}^D \times \mathbf{Gr}^+(D,d)$, the contribution from the Grassmannian factor enlarges the lengths of curves near points of large normal curvature, so that unit-speed geodesics have smaller normal curvature, and bottleneck pairs with opposite normal directions are pushed farther apart because their oriented tangent spaces are far from each other in $\mathbf{Gr}^+(D,d)$.

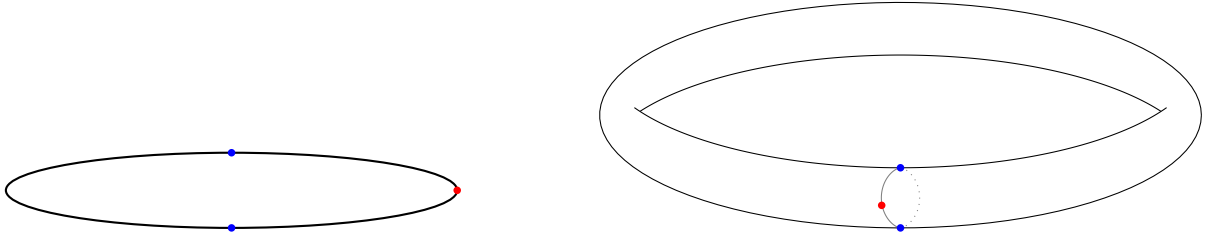
\begin{figure}
\centering
\begin{subfigure}[t]{0.47\textwidth}
\centering
\begin{tikzpicture}
  \draw[thick] (0,0) ellipse [x radius=3.0cm, y radius=0.5cm];
  \fill[blue] (0,0.5) circle(0.05);
  \fill[blue] (0,-0.5) circle(0.05);
  \fill[red] (3.0,0) circle(0.05);
\end{tikzpicture}
\end{subfigure}
\hfill
% ---- (b) Very thin torus (2D projection style) ----
\begin{subfigure}[t]{0.47\textwidth}
\centering
\begin{tikzpicture}[samples=100, variable=\t] 
  \def\a{4}
  \def\b{1.5}
  \shade[top color=gray!8,bottom color=gray!25]
    (0,0) ellipse [x radius=\a cm, y radius=\b cm];
  \path[fill=white]
    plot[domain=pi/6:5*pi/6] ({\a*cos(\t r)},{\b*sin(\t r) - 0.7})
    -- plot[domain=7*pi/6:11*pi/6] ({\a*cos(\t r)},{\b*sin(\t r) + 0.8})
    -- cycle;
  \draw[gray!80,domain=0:2*pi]      plot ({\a*cos(\t r)},{\b*sin(\t r)});
  \draw[gray!80,domain=pi/6:5*pi/6]
                                   plot ({\a*cos(\t r)},{\b*sin(\t r) - 0.7});
  \draw[gray!80,domain=7*pi/6:11*pi/6]
                                   plot ({\a*cos(\t r)},{\b*sin(\t r) + 0.8});
  \pgfmathsetmacro{\R}{0.8/(2*pi)}
  \draw[domain=0:2*pi,smooth,semithick,orange!65!black]
    plot ({-\R*(1-cos(\t r))}, {-\R*(\t - sin(\t r)) - 0.7});
  \draw[domain=0:2*pi,smooth,semithick,orange!65!black,densely dashed]
    plot ({\R*(1 - cos(\t r))},
        {-\R*(\t - sin(\t r)) - 0.7});

  \fill[blue] (0,-0.7) circle(0.05);
  \fill[blue] (0,-1.5) circle(0.05);
  \fill[red] (-0.25,-1.2) circle(0.05);
\end{tikzpicture}
\end{subfigure}
\caption{A highly eccentric ellipse (left) and a thin 2-torus (right). Points of large curvature are marked in red, and the narrowest bottlenecks are indicated in blue.}
\label{ellipse_torus}
\end{figure}

To formalize the above discussion, let $\iota : M \to \mathbb{R}^D$ be an embedding of a smooth closed orientable $d$-dimensional submanifold. Fix $c>0$ and consider the \emph{scaled oriented Grassmannian}  $(\mathbf{Gr}_c^+(D,d), g^{\mathbf{Gr}_c^+(D,d)})$, whose underlying manifold is the oriented Grassmannian $\mathbf{Gr}^+(D,d)$ equipped with the rescaled metric $g^{\mathbf{Gr}_c^+(D,d)} = c \cdot g^{\mathbf{Gr}^+(D,d)}$. Let $\bar{\mathbf{g}}^+$ denote the Gauss map, and define the rescaled Gauss map
$$\bar{\mathbf{g}}^+_c : M \xrightarrow{\bar{\mathbf{g}}^+} (\mathbf{Gr}^+(D,d), g^{\mathbf{Gr}^+(D,d)}) \xrightarrow{ \cdot c } (\mathbf{Gr}_c^+(D,d), g^{\mathbf{Gr}_c^+(D,d)}).$$
We consider the graph embedding $\iota \times \bar{\mathbf{g}}_c^+ : M \to (\mathbb{R}^D \times \mathbf{Gr}_c^+(D,d),\, g^{\mathbb{R}^D} \oplus g^{\mathbf{Gr}_c^+(D,d)}),$ and define $g_c = (\iota \times \bar{\mathbf{g}}^+_c)^*(g^{\mathbb{R}^D} \oplus g^{\mathbf{Gr}_c^+(D,d)})$. Throughout, we write $\mathbf{II}$ for the second fundamental form of $M \subset \mathbb{R}^D$ with the induced Euclidean metric and $\mathbf{II}_c$ for that of $(M,g_c) \subset \mathbb{R}^D \times \mathbf{Gr}_c^+(D,d)$ with the product metric. Likewise, we write $\mathrm{vol}_c(M)$ for the Riemannian volume of $(M,g_c)$ and $\mathrm{vol}(M)$ for the Euclidean volume of $M \subset \mathbb{R}^D$.
\begin{definition}\label{grassmannian_distance}
For $c>0$, we define a distance $d_c$ on $M$ by
\[
  d_{c}(q_1,q_2) = d_{\mathbb{R}^D \times \mathbf{Gr}_c^+(D,d)}\bigl((q_1, \bar{\mathbf{g}}_c^+(q_1)),(q_2, \bar{\mathbf{g}}_c^+(q_2))\bigr)
  = \sqrt{\|q_1-q_2\|_{\mathbb{R}^D}^2
      + c\, d_{\mathbf{Gr}^+(D,d)}\bigl(\bar{\mathbf{g}}^+(q_1),\bar{\mathbf{g}}^+(q_2)\bigr)^2},
\]
where $\|\cdot\|_{\mathbb{R}^D}$ denotes the Euclidean norm and
$d_{\mathbf{Gr}^+(D,d)}$ is the Riemannian distance on the oriented Grassmannian
$\mathbf{Gr}^+(D,d)$.
\end{definition}

We first show that, by embedding $M$ into $\mathbb{R}^D \times \mathbf{Gr}_c^+(D,d)$, we can reduce the normal curvature of $M$ along directions where the original normal curvature is large. For $p \in M$ and a nonzero vector $v \in T_pM$, denote by $\kappa_c(v)$ and $\kappa(v)$ the (operator norms of the) normal curvatures of $(M,g_c) \subset \mathbb{R}^D \times \mathbf{Gr}_c^+(D,d)$ at $(p,\bar{\mathbf{g}}_c^+(p))$ in the direction $(v, d\bar{\mathbf{g}}_c^+(v))$ and of $M \subset \mathbb{R}^D$ at $p$ in the direction $v$, respectively, where both geodesics are parametrized by arc length.
\begin{theorem}\label{principal_curvature}
  Let $M \subset \mathbb{R}^D$ be a smooth closed oriented $d$-dimensional submanifold. Fix $p \in M$, a unit vector $v \in T_p M$ such that $\mathbf{II}(v,\cdot)\neq 0$, and $c>0$. Assume that
  \[
    \kappa(v) > \frac{\|(\nabla_v^{M} \mathbf{II}) (v,\cdot) \|_{\operatorname{HS}}}{\|\mathbf{II} (v,\cdot)\|_{\operatorname{HS}}},
    \qquad
    \mathrm{vol}_c(M) \leq \bigl(1 + c \,\|\mathbf{II}(v,\cdot)\|_{\operatorname{HS}}^2\bigr)^{d/2} \mathrm{vol}(M),
  \]
  where $\|\cdot\|_{\operatorname{HS}}$ denotes the Hilbert--Schmidt norm. Then
  \[
    \kappa_c(v)\, \mathrm{vol}_c(M)^{1/d}
    \;\leq\;
    \sqrt{\frac{\kappa(v)^2 + c\|(\nabla_v^{M} \mathbf{II})(v, \cdot)\|_{\operatorname{HS}}^2}{1 + c\|\mathbf{II}(v,\cdot)\|_{\operatorname{HS}}^2}}\,
    \mathrm{vol}(M)^{1/d}
    \;<\;
    \kappa(v)\, \mathrm{vol}(M)^{1/d}.
  \]
  Moreover, the middle term, viewed as a function of $c>0$, is strictly decreasing.
\end{theorem}
In directions with large normal curvatures, even a small displacement along the curve already produces a large change in the tangent space, so the Grassmannian component of $g_c$ contributes an increasingly large portion of the speed. Since the normal curvature is defined as the size of the normal acceleration divided by the square of this speed, reparameterizing the curve to have unit speed with respect to $g_c$ yields an upper bound for $\kappa_c(v)\mathrm{vol}_c(M)^{1/d}$ that is strictly decreasing in $c$ under the hypotheses of Theorem~\ref{principal_curvature}.

\begin{definition}\label{bottleneck_widths}
For the Euclidean embedding, define
\[
  L(M)
  = \min\Bigl\{\frac{1}{2}\|q_1-q_2\|_{\mathbb{R}^D} \;\Big|\; (q_1,q_2) \text{ is a bottleneck of } M \subset \mathbb{R}^D\Bigr\}.
\]
For the lifted embedding, define
\[
\begin{aligned}
L_c'(M)=\min\Bigl\{\frac{1}{2}d_c(q_1,q_2)\ \Bigm|\ &q_1,q_2\in M,\ q_1\neq q_2,\text{ and a minimizing geodesic joining }(q_1,\bar{\mathbf g}_c^+(q_1))\\
&\text{and }(q_2,\bar{\mathbf g}_c^+(q_2))\text{ is orthogonal to }(M,g_c)\text{ at both endpoints}\Bigr\}.
\end{aligned}
\]
If no such minimizing geodesic segment exists, we set $L_c'(M)=+\infty$. 
\end{definition}
The following theorem shows that for sufficiently small $c>0$, the
lifted quantity $L_c'(M)$ is strictly larger than $L(M)$ after normalization by the corresponding volume.
\begin{theorem}\label{normalized_lifted_bottleneck}
Let $M\subset\mathbb{R}^D$ be a smooth closed connected oriented hypersurface. Suppose that $\mathsf{rch}_{\mathbb{R}^D}(M)=L(M)=\tau>0$. Then there exists $c_0>0$ such that, for every $c\in(0,c_0)$,
\[
\frac{L_c'(M)}{\mathrm{vol}_c(M)^{1/(D-1)}}
>
\frac{L(M)}{\mathrm{vol}(M)^{1/(D-1)}}.
\]
\end{theorem}

By deriving lower bounds for the reach of
$(M,g_c)\subset\mathbb{R}^D\times\mathbf{Gr}_c^+(D,d)$
and the convexity radius of the ambient product, we obtain an explicit
range of radii $r$ for which the ambient Čech complex
$\check{C}(M,\mathbb{R}^D\times\mathbf{Gr}_c^+(D,d);r)$
is homotopy equivalent to $M$. In particular, this yields a lower bound on the lengths of the bars in the associated persistent homology, from which the homology of $M$ is recovered from the distance $d_c$.
\begin{theorem}\label{length_barcodes}
Let $M \subset \mathbb{R}^D$ be a smooth closed connected orientable $d$-dimensional submanifold. For $c>0$, the ambient Čech complex $\check{C}(M, \mathbb{R}^D\times\mathbf{Gr}_c^+(D,d);r)$ is homotopy equivalent to \(M\) for all
\[
  0<r <
  \min\!\left(
    \sqrt{\frac{c}{2}}\arctan\sqrt{\frac{2}{c\|\mathbf{II}_c\|_2^2}},\,
    \frac{\sqrt{c}\,\pi}{4},\,
    L'_c(M)
  \right).
\]
\end{theorem}

Fix a field $\mathbb{k}$. Since the metric $d_c$ depends both on the ambient positions and on the oriented tangent spaces of $M$, the persistence module $H_j(\mathbb{\check{C}}(M, \mathbb{R}^D \times \mathbf{Gr}_c^+(D,d));\mathbb{k})$ is not expected to be stable with respect to the Hausdorff distance on $\mathbb{R}^D$. Instead, we work with a stronger notion of convergence that also takes into account the distance between tangent spaces.
\begin{theorem}\label{stability}
  Let $\{M_i\}_{i \in \mathbb{N}}$ be a sequence of smooth closed connected orientable $d$-dimensional submanifolds of $\mathbb{R}^D$ converging to a smooth closed connected orientable submanifold $M_\infty \subset \mathbb{R}^D$ in the Whitney $C^1$ topology (see Definition~\ref{Whitney}). Fix $c>0$. Then for every field $\mathbb{k}$ and every $j \in \mathbb{N} \cup \{0\}$,
  \[
    \lim_{i \to \infty}
    d_I\Bigl(
      H_j(\mathbb{\check{C}}(M_i, \mathbb{R}^D\times\mathbf{Gr}_c^+(D,d));\mathbb{k}),
      H_j(\mathbb{\check{C}}(M_\infty, \mathbb{R}^D\times\mathbf{Gr}_c^+(D,d));\mathbb{k})
    \Bigr) = 0,
  \]
  where $d_I$ denotes the interleaving distance between persistence modules.
\end{theorem}

For a $d$-dimensional manifold immersed in $\mathbb{R}^D$ with self-intersections, Tinarrage \cite{tinarrage2023recovering} lifts each sampled point to $\mathbb{R}^D \times \mathbf{Gr}(D,d)$ by adjoining a tangent space estimated from the local covariance matrix at that point, so that distinct preimages of a self-intersection point are separated whenever the corresponding tangent spaces in $\mathbb{R}^D$ differ. For a time series sampled from a smooth flow in $\mathbb{R}^D$, Bauer et al. \cite{bauer2023cycling} lift the time series to $\mathbb{R}^D \times S^{D-1}$ by pairing each sampled position with an estimated unit velocity, and compare two lifted points by taking the larger of their spatial distance and a weighted distance between their unit velocity directions, so that spatially close points moving in different directions remain separated. In \cite{kalisnik2020finding,kalivsnik2024persistent}, ellipsoids elongated in the estimated tangent directions are used in place of Euclidean balls in order to extend the filtration neighborhoods farther in the tangent directions than in the normal directions.

Section~\ref{sec:background} collects the background material. Section~\ref{sec:embedding} proves the main results by bounding the volume of the lift in Section~\ref{sec:volume}, estimating its normal curvature in Section~\ref{sec:curvature}, treating bottlenecks in Section~\ref{sec:bottleneck}, deriving a range of radii for homotopy equivalences in Section~\ref{sec:radii}, and proving stability in Section~\ref{sec:stability}. Section~\ref{sec:computation} explains how to approximate $d_c$ on a finite sample using local PCA, computes persistent homology from the resulting distance matrix, and illustrates the construction on a time-delay embedded attractor, an approximate quasi-halo orbit in the Saturn--Enceladus system, and a cortical ribbon extracted from an anatomical brain template.

\section{Theoretical background}\label{sec:background}

\subsection{Persistence theory}

A \emph{filtered space (or filtration)} $\mathbb{X}$ is a collection of topological spaces $\{X_i\}_{i \in \mathbb{R}_{\geq 0}}$ equipped with inclusion maps $\iota_i^j : X_i \to X_j$ for all $i \leq j$. Let \( (X, d) \) be a compact metric space and let \( r \in \mathbb{R}_{\geq 0} \). The (intrinsic) \emph{Čech complex} $\check{C}(X;r)$ is the nerve of the family of open $r$-balls $\{B_X(x,r)\}_{x\in X}$ centered at points of $X$. For $0\leq r\leq s$, the natural inclusions $\check{C}(X;r)\hookrightarrow\check{C}(X;s)$ form the \emph{Čech filtration}, denoted by $\mathbb{\check{C}}(X)$. Let $Y$ be a subset of a metric space $(Z,d_Z)$. The \emph{ambient Čech complex} $\check{C}(Y, Z; r)$ is defined as the nerve of the family of open $r$-balls $\{B_Z(y, r)\}_{y \in Y}$ taken in the ambient space $Z$. The corresponding filtration is denoted by $\mathbb{\check{C}}(Y, Z)$. For \( \varepsilon \in \mathbb{R}_{\geq 0} \), the \emph{Vietoris--Rips complex} \( \mathrm{VR}(X;\varepsilon) \) is the abstract simplicial complex whose simplices are finite subsets \( \sigma \subset X \) such that $d(x, y) < \varepsilon$ for all $x, y \in \sigma$. The corresponding \emph{Vietoris--Rips filtration} $\mathbb{V}R(X)$ is defined analogously.

For $n \in \mathbb{N} \cup \{0\}$, a field $\mathbb{k}$ and a filtration $\mathbb{X}= \{X_i\}_{i \in \mathbb{R}_{\geq 0}}$, the \emph{$n$-th persistent homology} $ H_n(\mathbb{X}; \mathbb{k}) $ of $\mathbb{X}$ is a family of vector spaces $H_n(X_i; \mathbb{k})$ with induced linear maps $H_n(\iota_i^j) : H_n(X_i; \mathbb{k}) \to H_n(X_j; \mathbb{k})$ for $i\leq j$. If \( H_n(\mathbb{X}; \mathbb{k}) \) is pointwise finite-dimensional and decomposes as a direct sum of interval modules, then the multiset of intervals in this decomposition is called the \emph{barcode} in degree $n$. The corresponding multiset of birth--death pairs $(b,d)$, where $d$ may be infinite, is called the \emph{persistence diagram} in degree $n$.

For a field $\mathbb{k}$, a \emph{persistence module (over $\mathbb{R}_{\geq 0}$)} $\mathbb{V}$ is a family of vector spaces $\{V_{i}\}_{i \in \mathbb{R}_{\geq 0}}$ equipped with linear maps $v_i^j :V_i \to V_j$ for every $0 \leq i \leq j$, satisfying $v_j^k \circ v_i^j = v_i^k$ whenever $i \leq j \leq k$ and $v_i^i$ is the identity map on $V_i$. A \emph{morphism of degree $\varepsilon$} between two persistence modules $\mathbb{V}$ and $\mathbb{W}$ is a family of linear maps $\Phi = \{\phi_i:V_i\to W_{i+\varepsilon}\}_{i\in\mathbb{R}_{\geq 0}}$ such that $w_{i+\varepsilon}^{j+\varepsilon}\circ \phi_i = \phi_j\circ v_i^j$ for all $i\le j$.
For $\varepsilon\ge 0$, the \emph{shift} $\Sigma^\varepsilon$ of $\mathbb{V}$ by $\varepsilon$ is
\[
(\Sigma^\varepsilon \mathbb{V})_i=V_{i+\varepsilon},\qquad
(\Sigma^\varepsilon)(v_i^j)=v_{i+\varepsilon}^{j+\varepsilon}:V_{i+\varepsilon} \to V_{j+\varepsilon}.
\]
Two persistence modules $\mathbb{V}$ and $\mathbb{W}$ are \emph{$\varepsilon$-interleaved} if there exist morphisms $\Phi=\{\phi_i : V_i \to W_{i+\varepsilon}\}_{i \in \mathbb{R}_{\geq 0}}$ and $\Psi = \{\psi_i : W_i \to V_{i+\varepsilon}\}_{i \in \mathbb{R}_{\geq 0}}$ of degree $\varepsilon$ such that
\[
\psi_{\,i+\varepsilon}\circ \phi_i \;=\; v_{i}^{i+2\varepsilon},
\qquad
\phi_{\,i+\varepsilon}\circ \psi_i \;=\; w_{i}^{i+2\varepsilon},
\]
for every $i \in \mathbb{R}_{\geq 0}$. The \emph{interleaving distance} $d_I$ between $\mathbb{V}$ and $\mathbb{W}$ is
\[
d_I(\mathbb{V},\mathbb{W})\;=\;\inf\bigl\{\varepsilon\ge 0 \,\big|\, \mathbb{V} \text{ and } \mathbb{W} \text{ are } \varepsilon\text{-interleaved}\bigr\}.
\]
For more details, see e.g.\ \cite{oudot2015persistence, chazal2021introduction}.

\subsection{Basic Riemannian geometry}

In this subsection, we refer to \cite{do1992riemannian} and \cite{lee2018introduction}. Let $M$ be a smooth compact $d$-dimensional manifold. For a smooth vector bundle $E\to M$, let $\Gamma(E)$ denote the space of smooth sections of $E$ and set $\mathfrak{X}(M)=\Gamma(TM)$. A smooth section $g \in \Gamma(\mathrm{Sym}^2 T^*M)$ is called a \emph{Riemannian metric} if, for each $p \in M$, the bilinear form $g_p : T_pM \times T_pM \to \mathbb{R}$ is positive definite. The pair $(M,g)$ is then called a \emph{Riemannian manifold}. We call the map 
$$\nabla : \mathfrak{X}(M) \times \mathfrak{X}(M) \to \mathfrak{X}(M); \quad (X,Y) \mapsto \nabla_X Y$$
an \emph{affine connection} if $\nabla_X Y$ is linear over $C^\infty (M)$ in $X$ and linear over $\mathbb{R}$ in $Y$, satisfying $\nabla_X(fY) = f(\nabla_X Y) + (Xf)Y$ for any $f \in C^\infty(M)$. The \emph{Levi--Civita connection} is the unique affine connection $\nabla$ on $(M,g)$ that satisfies $Z(g(X,Y)) =g(\nabla_Z X, Y) + g(X,\nabla_Z Y)$ and $\nabla_X Y - \nabla_Y X = [X,Y]$ for any $X,Y,Z \in \mathfrak{X}(M)$.
The connection is extended to arbitrary tensor fields by requiring the Leibniz rule, commutation with contractions, and agreement with directional derivatives on functions.

A smooth curve $\gamma$ in $M$ is a \emph{geodesic} if $\nabla_{\gamma'}\gamma'=0$. For $p\in M$, the \emph{Riemannian exponential map} $\exp_p:T_pM\to M$ is defined by $\exp_p(v)=\gamma_v(1)$, where $\gamma_v$ is the geodesic of $M$ satisfying $\gamma_v(0)=p$ and $\gamma_v'(0)=v$. The \emph{injectivity radius} $\operatorname{inj}(M,g)$ is the infimum over $p\in M$ of the supremum of the radii of centered balls in $T_pM$ on which $\exp_p$ is a diffeomorphism onto its image. The \emph{convexity radius} of $(M,g)$ is the supremum of all $r>0$ such that every geodesic ball in $M$ of radius less than $r$ is strongly convex. The \emph{systole} $\mathrm{sys}(M,g)$ is the infimum of the lengths of piecewise smooth non-contractible loops in $M$.

For a Riemannian manifold $(M,g)$ with Levi--Civita connection $\nabla$, the \emph{Riemann curvature tensor} $R$ is defined by $R(X,Y)Z = \nabla_X \nabla_Y Z - \nabla_Y \nabla_X Z- \nabla_{[X,Y]}Z$ for $X,Y,Z \in \mathfrak{X}(M)$. For a tangent $2$-plane $\Sigma_pM \subset T_pM$ with an orthonormal basis $\{u,v\}$, the \emph{sectional curvature} of $\Sigma_p M$ is defined by $K_p(u,v)=g_p(R(u,v)v,u).$

\subsection{Convergence theory}
We refer to \cite{gromov2007metric, hirsch2012differential} in this subsection. Let $(Z,d_Z)$ be a metric space and let $A,B\subset Z$ be nonempty compact subsets. The \emph{Hausdorff distance} between $A$ and $B$ is defined by
\[
d_H^Z(A,B)
=\max\Big\{\,\sup_{a\in A}\inf_{b\in B} d_Z(a,b)\;,\;
\sup_{b\in B}\inf_{a\in A} d_Z(a,b)\,\Big\}.
\]
The following stability theorem establishes the continuity of persistent homology of the ambient Čech complex with respect to the Hausdorff distance.
\begin{theorem}[{\cite[Corollary~4.10]{chazal2014persistence}, adapted to open balls}]\label{stability_theorem} Let $A,B$ be compact subsets of a metric space $Z$. Then for every $j \in \mathbb{N} \cup \{0\}$ and every field $\mathbb{k}$,
\[
d_I\bigl(H_j(\mathbb{\check{C}}(A, Z); \mathbb{k}), H_j(\mathbb{\check{C}}(B, Z); \mathbb{k})\bigr)
\leq d_H^Z(A,B).
\]
\end{theorem}

Let $(N,g_N)$ be a fixed complete Riemannian manifold, and let $\{M_i\}_{i\in\mathbb{N}}$ be a family of smooth closed submanifolds of $N$ such that $M_i = F_i(M)$ for every $i \in \mathbb{N}$, where $M$ is a fixed smooth closed manifold and each $F_i : M \to N$ is a smooth embedding.
\begin{definition}[{\cite[Chapter~2]{hirsch2012differential}}]\label{Whitney}
We say that the family of submanifolds $\{M_i\}_{i\in\mathbb{N}}$ \emph{converges to $M_\infty$ in the (Whitney) $C^k$ topology} if there exists an embedding
$$
F_\infty : M \to N, \qquad M_\infty = F_\infty(M),
$$
such that $F_i \to F_\infty$ in the Whitney $C^k$ topology on $C^k(M,N)$. That is, there exist finitely many coordinate charts $\{(U_\alpha,\varphi_\alpha)\}_{\alpha=1}^m$ of $M$, coordinate charts $\{(V_\alpha,\psi_\alpha)\}_{\alpha=1}^m$ of $N$, and compact subsets $K_\alpha\subset U_\alpha$ such that
\begin{itemize}
    \item the interiors of the sets $K_\alpha$ cover $M$.
    \item the inclusions $F_\infty(K_\alpha)\subset V_\alpha$ and $F_i(K_\alpha)\subset V_\alpha$ hold for every $\alpha$ and all sufficiently large $i$.
    \item as $i\to\infty$, we have
\[
\max_{1\leq\alpha\leq m}\max_{|\beta|\le k}
\sup_{x \in \varphi_\alpha(K_\alpha)}
\left\|
D^\beta\big(\psi_\alpha \circ F_i \circ \varphi_\alpha^{-1}\big)(x)
-
D^\beta\big(\psi_\alpha \circ F_\infty \circ \varphi_\alpha^{-1}\big)(x)
\right\|_2
\longrightarrow 0.
\]
\end{itemize}
\end{definition}

\subsection{Geometry of an embedded submanifold}

Let $M \subset \mathbb{R}^D$ be a smooth closed $d$-dimensional submanifold with $d < D$. Two distinct points $q_1,q_2 \in M$ form a \emph{bottleneck} in $\mathbb{R}^D$ if the line segment $\overline{q_1 q_2}$ is orthogonal to both tangent spaces $T_{q_1}M$ and $T_{q_2}M$. The \emph{width} of such a bottleneck is defined as $\tfrac{1}{2}\|q_1-q_2\|_{\mathbb{R}^D}$. This notion should not be confused with the bottleneck distance between persistence modules 
or with bottlenecks in graph theory.

Define the (Euclidean) \emph{reach} of $M \subset \mathbb{R}^D$ by
\[
\textsf{rch}_{\mathbb{R}^D}(M)
= \sup \Bigl\{ r \geq 0 \,\Big|\,
\text{every } p \in \mathbb{R}^D \text{ with } d_{\mathbb{R}^D}(p,M)< r
\text{ admits a unique nearest point projection onto } M \Bigr\}.
\]
\begin{theorem}[{\cite[Theorem~3.4]{aamari2019estimating}, \cite[Theorem~7.8]{breiding2024metric}}]\label{reach}
    Suppose that the reach of a closed submanifold $M \subset \mathbb{R}^D$ is $\tau>0$. Then at least one of the following holds.
    \begin{itemize}
        \item There exist a point $q \in M$ and a unit-speed geodesic $\gamma : (-\varepsilon, \varepsilon) \to M$ for some $\varepsilon>0$ such that $\gamma(0)=q$ and $|\gamma''(0)| = \frac{1}{\tau}$.
        \item There exist distinct points $q_1, q_2 \in M$ forming a bottleneck of $M \subset \mathbb{R}^D$ such that $\|q_1 - q_2\|_{\mathbb{R}^D} = 2\tau$.
    \end{itemize}
    In particular, the reach $\tau$ of $M$ is realized either as the reciprocal of the normal curvature of a geodesic or as the width of a bottleneck pair.
\end{theorem}

Suppose a $d$-dimensional smooth compact manifold $M$ is embedded in a complete Riemannian manifold $(N,g)$. Denote by $\nabla$ and $\nabla^M$ the Levi--Civita connections of $(N,g)$ and of the induced metric on $M$, respectively, and by $\nu$ the normal bundle of $M$. The \emph{second fundamental form} of $M$ is defined to be the tensor $\mathbf{II} : \mathfrak{X}(M) \times \mathfrak{X}(M) \to \Gamma(\nu)$ given by $\mathbf{II}(X,Y) = (\nabla_X Y)^{\perp}$, where the map $(\cdot)^{\perp}$ denotes the orthogonal projection onto $\nu$. The \emph{normal exponential map} $\exp_\nu : \nu \to N$ is defined by
$$ \exp_\nu(q,v) = \exp_q(v),\quad (q,v) \in \nu.$$
For $p \in M$ and a nonzero vector $v \in T_pM$, the \emph{(norm of the) normal curvature} $\kappa(v)$ at $p$ in the direction $v$ is 
\[
\kappa(v)
=
\sup_{\substack{w \in \nu_p \\ \|w\|_2=1}}
\frac{g_p\bigl(\mathbf{II}_p(v,v),\, w\bigr)}{\|v\|_2^2}.
\]
The \emph{operator norm} (or \emph{2-norm}) of $\mathbf{II}_p$ is
\[
\|\mathbf{II}_p\|_2
=
\sup_{\substack{v \in T_pM \\ \|v\|_2=1}}
\kappa(v)
=
\sup_{\substack{v \in T_pM \\ \|v\|_2=1}}
\bigl\|\mathbf{II}_p(v,v)\bigr\|_2.
\]
Denote $\|\mathbf{II}\|_2 = \sup_{p \in M} \|\mathbf{II}_p\|_2$. For two linear operators $A_p,B_p : T_pM \to \nu_p$ and an orthonormal basis $\{e_i\}_{i=1}^d$ of $T_pM$, the \emph{Hilbert--Schmidt inner product} between them is 
$$\langle A_p, B_p \rangle_{\operatorname{HS}} =\sum_{i=1}^d g_p\bigl(A_p(e_i), B_p(e_i)\bigr).$$
Denote by $\|\cdot\|_{\operatorname{HS}}$ the induced norm.

\begin{definition}[{\cite[Definition~3.1.1 and 3.1.2]{prasad2023cut}, \cite[Definitions~12 and 13]{attali2022tight}}]
Let $S$ be a closed subset embedded in a complete Riemannian manifold $N$. Denote the distance function on $N$ by $d_N$ and the length of a curve $\gamma \subset N$ by $\mathrm{len}_N(\gamma)$. We define a geodesic $\gamma : [0,T] \to N$ to be ($S$-)\emph{distance-minimal} if $\mathrm{len}_N(\gamma|_{[0,t]}) = d_N(S,\gamma(t))$ for all $t \in [0,T]$. The \emph{cut locus} of $S$, denoted by $\mathsf{Cu}_N(S)$, is the set of points $p \in N$ for which there exists a distance-minimal geodesic $\gamma$ from $S$ to $p$ such that any extension of $\gamma$ beyond its endpoint $p$ is not distance-minimal. The \emph{(cut locus) reach} (or \emph{normal injectivity radius}) of $S$ in $N$ is
$$\mathsf{rch}_N(S) = \inf\{ d_N(p,q) \mid q \in S,\ p \in \mathsf{Cu}_N(S)\}.$$
\end{definition}
For a smooth closed submanifold $M\subset N$ with normal bundle $\nu$, the quantity $\mathsf{rch}_N(M)$ is the supremum of all $\varepsilon>0$ such that the restriction of the normal exponential map
\[
\exp_\nu : \{(p,v)\in \nu \mid \|v\|_2<\varepsilon\} \longrightarrow N
\]
is an embedding. We state an analogue of Theorem~\ref{reach} below.
\begin{theorem}[{\cite[Section~2]{singh1988closest}, \cite[Lemma~A.2]{basu2023connection}}]\label{cutlocus}
    Let $M$ be a smooth closed submanifold of a complete Riemannian manifold $(N,g)$, and suppose $\mathsf{rch}_N(M)=T>0$. Denote the normal bundle of $M$ by $\nu$ and its unit normal bundle by $S(\nu)$. Then one of the following holds (see Figure~\ref{cutlocus_figure}):
    \begin{itemize}
        \item[\textnormal{(focal)}] There exists a pair $(p,v) \in S(\nu)$ such that the differential of the normal exponential map $d(\exp_\nu)_{(p,Tv)}$ is not of full rank, whereas $d(\exp_\nu)_{(p',tv')}$ has full rank for every $(p',v') \in S(\nu)$ and every $t$ with $0 <t<T$. In this case, the point $\exp_\nu(p,Tv)$ is called a \emph{first focal point} and $T$ a \emph{first focal time}.
        \item[\textnormal{(non-focal)}] There exist at least two distinct unit-speed distance-minimal geodesics $\gamma_1, \gamma_2 : [0,T] \to N$ from (not necessarily distinct) two points in $M$ satisfying $\gamma_1(T) = \gamma_2(T)$ and $ \gamma_i'(0) \perp TM$ for $i=1,2$.
    \end{itemize}
\end{theorem}

\begin{figure}[ht]
\begin{center}
\begin{tikzpicture}[scale=1.5]

  \begin{scope}[yshift=-0.8cm]
    \draw[thick,black,domain=-1.1:1.1,samples=200] plot (\x,{\x*\x});
    \fill[red] (0,0.5) circle(0.03);
    \fill[blue] (0,0) circle(0.03);
    \draw[gray,->] (0,0) -- (0,0.5);

    \pgfmathsetmacro{\aone}{0.12}
    \pgfmathsetmacro{\atwo}{-0.10}
    \pgfmathsetmacro{\xone}{\aone}
    \pgfmathsetmacro{\yone}{\aone*\aone}
    \pgfmathsetmacro{\xtwo}{\atwo}
    \pgfmathsetmacro{\ytwo}{\atwo*\atwo}

    \fill[blue] (\xone,\yone) circle(0.03);
    \fill[blue] (\xtwo,\ytwo) circle(0.03);
    \draw[gray,->] (\xone,\yone) -- (0,0.5);
    \draw[gray,->] (\xtwo,\ytwo) -- (0,0.5);
  \end{scope}

  % ===== CENTER: S^2 scene =====
  \begin{scope}[xshift=3.2cm]
    \def\az{25}\def\el{35}
    \pgfmathsetmacro{\caz}{cos(\az)} \pgfmathsetmacro{\saz}{sin(\az)}
    \pgfmathsetmacro{\cel}{cos(\el)} \pgfmathsetmacro{\sel}{sin(\el)}

    \shade[inner color=white,outer color=gray!18] (0,0) circle (1);
    \draw[gray] (0,0) circle (1);
    \draw[lightgray,dashed] (-1,0)
      arc[start angle=180,end angle=0,x radius=1,y radius=0.28];
    \draw[lightgray] (-1,0)
      arc[start angle=180,end angle=360,x radius=1,y radius=0.28];

    \def\ang{30}\def\hw{16}

    \pgfmathsetmacro{\cang}{cos(\ang)} \pgfmathsetmacro{\sang}{sin(\ang)}
    \pgfmathsetmacro{\PX}{\cang*\caz - \sang*\saz}
    \pgfmathsetmacro{\PY}{(\cang*\saz + \sang*\caz)*\cel}
    \pgfmathsetmacro{\QX}{-\PX} \pgfmathsetmacro{\QY}{-\PY}
    \fill[blue] (\PX,\PY) circle (0.03);
    \fill[red]  (\QX,\QY) circle (0.03);

    \draw[blue!60!black,domain=0:150,samples=160,variable=\x]
      plot ({cos(\x)*(\cang*\caz - \sang*\saz)},
            {cos(\x)*(\cang*\saz + \sang*\caz)*\cel - sin(\x)*\sel});
    \draw[blue!60!black,densely dashed,domain=150:180,samples=40,variable=\x]
      plot ({cos(\x)*(\cang*\caz - \sang*\saz)},
            {cos(\x)*(\cang*\saz + \sang*\caz)*\cel - sin(\x)*\sel});
    \draw[orange!85!black,densely dashed,domain=180:330,samples=160,variable=\x]
      plot ({cos(\x)*(\cang*\caz - \sang*\saz)},
            {cos(\x)*(\cang*\saz + \sang*\caz)*\cel - sin(\x)*\sel});
    \draw[orange!85!black,domain=330:360,samples=40,variable=\x]
      plot ({cos(\x)*(\cang*\caz - \sang*\saz)},
            {cos(\x)*(\cang*\saz + \sang*\caz)*\cel - sin(\x)*\sel});

    \draw[blue!60!black,->]
      ({cos(70)*(\cang*\caz - \sang*\saz)},
       {cos(70)*(\cang*\saz + \sang*\caz)*\cel - sin(70)*\sel})
      --
      ({cos(100)*(\cang*\caz - \sang*\saz)},
       {cos(100)*(\cang*\saz + \sang*\caz)*\cel - sin(100)*\sel});

    \draw[orange!85!black,densely dashed,->]
      ({cos(270)*(\cang*\caz - \sang*\saz)},
       {cos(270)*(\cang*\saz + \sang*\caz)*\cel - sin(270)*\sel})
      --
      ({cos(240)*(\cang*\caz - \sang*\saz)},
       {cos(240)*(\cang*\saz + \sang*\caz)*\cel - sin(240)*\sel});
  \end{scope}

  % ===== RIGHT: Hyperbola =====
  \begin{scope}[xshift=6.4cm,scale=0.7]
    \draw[thick,black,domain=-1.3:1.3,samples=300]
      plot ({ sqrt(1+\x*\x) },{\x});
    \draw[thick,black,domain=-1.3:1.3,samples=300]
      plot ({-sqrt(1+\x*\x) },{\x});

    \fill[blue] ( 1,0) circle(0.03);
    \fill[blue] (-1,0) circle(0.03);
    \fill[red] (0,0) circle(0.03);

    \draw[gray,->] ( 1,0) -- (0,0);
    \draw[gray,->] (-1,0) -- (0,0);
  \end{scope}
\end{tikzpicture}
\end{center}
\caption{Examples of cut-locus points (in red): 
a first focal point caused by large normal curvature (left), 
a non-focal cut locus arising from two distance-minimal geodesics from the same point (middle), 
and a non-focal cut locus arising from two distinct points (right).}
\label{cutlocus_figure}
\end{figure}
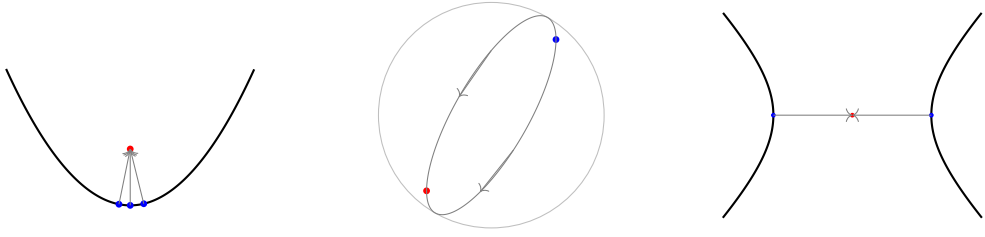
We will establish a lower bound on the reach of the embedding $(M,g_c) \subset \mathbb{R}^D \times \mathbf{Gr}_c^+(D,d)$ by deriving casewise lower bounds. This gives a range of radii $r$ for which the ambient Čech complex $\check{C}(M, \mathbb{R}^D \times \mathbf{Gr}_c^+(D,d); r)$ is homotopy equivalent to $M$.

\begin{theorem}[{\cite[Theorem~A]{basu2023connection}}]\label{grad_flow}
Denote the function $f = d_N(\cdot, M)^2 : N \to \mathbb{R}$. Then $f$ is a Morse--Bott function on $N \setminus \mathsf{Cu}_N(M)$ with critical manifold $M$, and the set $N \setminus \mathsf{Cu}_N(M)$ deformation retracts onto $M$ via the negative gradient flow of $f|_{N \setminus \mathsf{Cu}_N(M)}$.
\end{theorem}

\begin{theorem}[{\cite[Proposition~16]{attali2022tight}, modified}]\label{niyogi_cutlocus}
Let $M$ be a smooth closed submanifold of a complete Riemannian manifold $N$ with non-negative sectional curvature. If a finite subset $\{x_i\}_{i=1}^n \subset M$ is $\varepsilon/2$-dense in $M$ for some $\varepsilon < \sqrt{6 - 4\sqrt{2}}\;\mathsf{rch}_{N}(M)$, then there exists a nonempty interval $I$ of radii $r$ such that the set $\bigcup_{i=1}^n B_N (x_i, r)$ deformation retracts onto $M$ for every $r \in I$. Moreover, a lower bound for the length of this interval depends only on $\varepsilon$ and $\mathsf{rch}_{N}(M)$.
\end{theorem}

\subsection{Geometry of the Grassmannian}
We mainly follow \cite{bendokat2024grassmann} and \cite{lai2025simple}. The compact manifold $\mathbf{Gr}(D,d)$, called the \emph{(real) Grassmannian}, is the space of all $d$-dimensional linear subspaces of $\mathbb{R}^D$ equipped with the quotient topology. For two elements $P, Q \in \mathbf{Gr}(D,d)$, let $A,B$ be matrices whose columns form orthonormal bases for $P$ and $Q$, respectively. Let $\sigma_1, \sigma_2, \dots, \sigma_d$ denote the singular values of $A^\top B$. The \emph{principal angles} between $P$ and $Q$ are defined by
$$(\theta_1,\dots,\theta_d) = (\arccos{\sigma_1},\dots, \arccos{\sigma_d}).$$
The Grassmannian $\mathbf{Gr}(D,d)$ carries a Riemannian metric whose induced geodesic distance is
$$d_{\mathbf{Gr}(D,d)}(P, Q) = \left(\sum_{i=1}^d \theta^2_i\right)^{\frac{1}{2}}.$$
Analogously, the space of $d$-dimensional oriented subspaces of $\mathbb{R}^D$ is called the \emph{(real) oriented Grassmannian}, denoted by $\mathbf{Gr}^+(D,d).$ For $1\leq d<D$, the map $\mathbf{Gr}^+(D,d)\to\mathbf{Gr}(D,d)$ that forgets the orientation is a connected double Riemannian covering and hence a local isometry. For two oriented subspaces $P, Q \in \mathbf{Gr}^+(D,d)$ with orthonormal basis matrices $A$ and $B$ compatible with the orientations of $P$ and $Q$, respectively, the \emph{geodesic distance} between $P$ and $Q$ is given by
\[
d_{\mathbf{Gr}^+(D,d)}(P, Q) =
\begin{cases}
\left(\displaystyle\sum_{i=1}^d \theta_i^2\right)^{\frac{1}{2}} & \text{if } \det(A^\top B) > 0,\\[0.6em]
\left(\displaystyle(\pi-\theta_1)^2 + \sum_{i=2}^{d} \theta_i^2 \right)^{\frac{1}{2}} & \text{otherwise},
\end{cases}
\]
where $\theta_1, \dots, \theta_d$ are the principal angles between $P,Q$ with $\theta_1 \geq \dots \geq \theta_d \geq 0$. 

For a smooth closed oriented $d$-dimensional submanifold $M \subset \mathbb{R}^D$, the (generalized) \emph{Gauss map} $\bar{\mathbf{g}}^+ : M \to \mathbf{Gr}^+(D,d)$ is defined by the composition
$$p \longmapsto T_pM \xrightarrow{\;\cong\;} P \in \mathbf{Gr}^+(D,d),$$
where the isomorphism $\xrightarrow{\;\cong\;}$ identifies the oriented tangent space $T_pM \subset \mathbb{R}^D$ with the corresponding $d$-dimensional oriented subspace of $\mathbb{R}^D$ translated to the origin. The \emph{unoriented Gauss map} $\bar{\mathbf{g}}: M \to \mathbf{Gr}(D,d)$ is defined in the same way, ignoring the orientation of the tangent spaces.

\section{Embedding into the oriented Grassmannian bundle}\label{sec:embedding}

\subsection{Bound for Riemannian volume}\label{sec:volume}

Every subspace $Q \in \mathbf{Gr}(D,d)$ sufficiently close to $P \in \mathbf{Gr}(D,d)$ is the graph $Q=\{v+L_Q(v)\mid v\in P\}$ of a unique linear map $L_Q:P\to P^\perp$ (Figure~\ref{grassmannian_graph}). This gives the standard identification $T_P\mathbf{Gr}(D,d) \cong \operatorname{Hom}(P,P^\perp)$, under which the Riemannian metric agrees with the Hilbert--Schmidt inner product \cite[Sections~2.3 and~3.1]{bendokat2024grassmann}.

\begin{figure}[ht]
\centering
\begin{tikzpicture}[
  x={(1cm,-0.28cm)},
  y={(0.35cm,0.32cm)},
  z={(0cm,0.95cm)},
  scale=1.05,
  line join=round,
  >=Stealth
]
  \coordinate (O) at (0,0,0);
  \coordinate (v) at (0.95,0.35,0);
  \coordinate (q) at (0.95,0.35,0.5225);

  \filldraw[fill=gray!18,draw=gray!65]
    (-1.7,-1.1,0) -- (1.7,-1.1,0) -- (1.7,1.1,0) -- (-1.7,1.1,0) -- cycle;
  \fill[blue!12]
    (-1.7,-1.1,-0.935) -- (1.7,-1.1,0.935) -- (1.7,1.1,0.935) -- (-1.7,1.1,-0.935) -- cycle;
  \fill[green!12]
    (-1.7,1.1,-0.935) -- (0,1.1,0) -- (0,-1.1,0) -- cycle;
  \draw[blue!45!gray]
    (-1.7,-1.1,-0.935) -- (1.7,-1.1,0.935) -- (1.7,1.1,0.935) -- (-1.7,1.1,-0.935) -- cycle;

  \draw[gray!75,densely dashed] (0,-1.1,0) -- (0,1.1,0);
  \draw[gray!75,densely dashed] (-1.7,1.1,-0.935) -- (0,-1.1,0);
  \draw[gray!70,densely dashed] (0,0,-1.2) -- (0,0,1.35)
    node[above] {$P^\perp$};
  \draw[->,very thick,brown!75!black] (O) -- (v)
    node[midway,below] {$v$};
  \draw[->,very thick,teal!70!black] (v) -- (q)
    node[midway,right] {$L_Q(v)$};
  \draw[->,thick,blue!65!black] (O) -- (q)
    node[pos=0.68,above,xshift=12pt,yshift=4pt] {$v+L_Q(v)$};

  \fill (O) circle (1.3pt);
  \fill[brown!75!black] (v) circle (1.3pt);
  \fill[blue!65!black] (q) circle (1.3pt);
  \node[gray!70!black] at (1.5,-0.9,0) {$P$};
  \node[blue!55!black,yshift=-5pt] at (-1.25,0.9,-0.78) {$Q$};
\end{tikzpicture}
\caption{The graph representation of a subspace $Q$ near $P$. The vector $v+L_Q(v)\in Q$ decomposes into $v\in P$ and $L_Q(v)\in P^\perp$.}
\label{grassmannian_graph}
\end{figure}

\begin{samepage}
\begin{lemma}[{\cite[Section~3]{lai2025simple} and \cite[Section~2.1]{nagatomo2021holomorphic}}]\label{grass_tangent}
    Let $M \subset \mathbb{R}^D$ be a smooth closed $d$-dimensional submanifold with normal bundle $\nu$. Denote the unoriented Gauss map by $\bar{\mathbf{g}} : M \to \mathbf{Gr}(D,d)$. For any $X,Y \in \mathfrak{X}(M)$ and $p \in M$,
    $$ d\bar{\mathbf{g}}_p(X_p)
= \mathbf{II}_p(X_p,\cdot) \in \operatorname{Hom}(T_pM, \nu_p),$$
    where $\mathbf{II}$ denotes the second fundamental form of $M \subset \mathbb{R}^D$. Moreover,
    $$g^{\mathbf{Gr}(D,d)}_{\bar{\mathbf{g}}(p)}
\bigl(d\bar{\mathbf{g}}_p(X_p), d\bar{\mathbf{g}}_p(Y_p)\bigr)
=
\langle \mathbf{II}_p(X_p,\cdot), \mathbf{II}_p(Y_p,\cdot) \rangle_{\operatorname{HS}}.$$
    Under the identification $\bar{\mathbf{g}}^*T\mathbf{Gr}(D,d) \cong \operatorname{Hom}(TM,\nu)$, the pullback along $\bar{\mathbf{g}}$ of the Levi--Civita connection on $\mathbf{Gr}(D,d)$ is induced by the Levi--Civita connection on $TM$ and the normal connection on $\nu$. More explicitly, let $A$ be a section of $\operatorname{Hom}(TM,\nu)$. For $X,Z\in\mathfrak{X}(M)$, the pullback connection satisfies
    \[
    \bigl((\bar{\mathbf{g}}^*\nabla^{\mathbf{Gr}(D,d)})_X A\bigr)(Z)
    =(\nabla_X^M A)(Z).
    \]
\end{lemma}
\end{samepage}

\begin{theorem}\label{volume}
  Let $M \subset \mathbb{R}^D$ be a smooth closed oriented $d$-dimensional submanifold, and let $\mathbf{II}$ be the second fundamental form of $M$ with $\|\mathbf{II}\|_{2}=a$. Let $\mathrm{vol}(M)$ denote the Euclidean volume of $M$. For $c>0$, let $\mathrm{vol}_c(M)$ denote the Riemannian volume of $(M,g_c)$. Then
\[
  \mathrm{vol}_c(M)
  \;\leq\;
  \Bigl(
    1
    +
    c\Bigl(
      \sup_{p \in M}\;\sup_{\substack{v \in T_pM\\ \|v\|_2=1}} 
      \|\mathbf{II}_p(v,\cdot)\|_{\operatorname{HS}}
    \Bigr)^2
  \Bigr)^{d/2}
  \mathrm{vol}(M).
  \]
In particular,
  $$\mathrm{vol}_c(M) \leq (1+ca^2 \min (d, D-d) )^{d/2} \mathrm{vol}(M).$$
\end{theorem}

\begin{proof}
By Lemma~\ref{grass_tangent} and the fact that $\mathbf{Gr}^+(D,d) \to \mathbf{Gr}(D,d)$ is a local isometry, for every $p \in M$ and $X \in \mathfrak{X}(M)$,
\[
(g_c)_p(X,X)
= \|X_p\|_2^2 + c\,\|\mathbf{II}_p(X_p,\cdot)\|_{\operatorname{HS}}^2.
\]
Set $K = \sup_{p \in M}\;\sup_{\substack{v \in T_pM\\ \|v\|_2=1}} \|\mathbf{II}_p(v,\cdot)\|_{\operatorname{HS}}$. Then $(g_c)_p(X,X) \leq (1+cK^2)\|X_p\|_2^2$.
Hence, the associated volume forms satisfy
\begin{equation}\label{volume_form_bound}
d\mathrm{vol}_c \;\le\; (1+cK^2)^{d/2}\,d\mathrm{vol}.
\end{equation}
Integrating over $M$ yields the first inequality. To obtain the second inequality, fix $p\in M$ and choose orthonormal bases $\{e_i\}_{i=1}^d$ of $T_pM$ and $\{n_\alpha\}_{\alpha=1}^{D-d}$ of $\nu_p$. For every unit vector $n\in\nu_p$, the map $(u,w)\mapsto\langle\mathbf{II}_p(u,w),n\rangle$ is a symmetric bilinear form. Hence,
\[
\sup_{\substack{\|u\|_2=1\\ \|w\|_2=1}}
\bigl|\langle \mathbf{II}_p(u,w),n\rangle\bigr|
=
\sup_{\|u\|_2=1}
\bigl|\langle \mathbf{II}_p(u,u),n\rangle\bigr|
\leq \|\mathbf{II}_p\|_2.
\]
It follows that
\[
\begin{aligned}
\|\mathbf{II}_p(X_p,\cdot)\|_{\operatorname{HS}}^2
&=\sum_{i=1}^d\|\mathbf{II}_p(X_p,e_i)\|_2^2
\leq d\,\|\mathbf{II}_p\|_2^2\|X_p\|_2^2,\\
\|\mathbf{II}_p(X_p,\cdot)\|_{\operatorname{HS}}^2
&=\sum_{\alpha=1}^{D-d}\sum_{i=1}^d
\bigl|\langle\mathbf{II}_p(X_p,e_i),n_\alpha\rangle\bigr|^2
\leq (D-d)\,\|\mathbf{II}_p\|_2^2\|X_p\|_2^2.
\end{aligned}
\]
Combining these estimates and taking the supremum over $p\in M$ and unit vectors $X_p\in T_pM$ shows that $K^2\leq \min(d,D-d)a^2$. Substituting this into \eqref{volume_form_bound} and integrating over $M$ yields the second inequality.
\end{proof}

\subsection{Normal curvature reduction}\label{sec:curvature}
The proof of Theorem~\ref{principal_curvature} has two preliminary steps. We first describe the normal bundle of the graph embedding and then compute the Grassmannian component of the ambient acceleration in terms of the second fundamental form and its covariant derivative. These results provide an estimate for the normal component of the ambient acceleration. Dividing this estimate by the squared speed with respect to $g_c$ and applying the volume assumption proves the first inequality in Theorem~\ref{principal_curvature}. The curvature assumption then implies the strict inequality and the strict monotonicity in $c$.

\begin{lemma}\label{normal_bundle}
Let $M \subset \mathbb{R}^D$ be a smooth closed oriented $d$-dimensional submanifold with normal bundle $\nu$. Denote by $\nu_c$ the normal bundle of $(M,g_c) \subset \mathbb{R}^D \times \mathbf{Gr}_c^+(D,d)$. Then for each $p \in M$, 
\[
(\nu_c)_p
=
\bigl\{
(w,0) + (-cS_p(Q),Q)
\;\big|\;
w \in \nu_p,\;
Q \in T_{\bar{\mathbf{g}}_c^+(p)} \mathbf{Gr}_c^+(D,d)
\bigr\}
\subset T_{(p, \bar{\mathbf{g}}_c^+(p))}(\mathbb{R}^D \times \mathbf{Gr}_c^+(D,d)),
\]
where $S_p(Q) \in T_pM$ is the unique vector characterized by $\langle v, S_p(Q) \rangle_{\mathbb{R}^D}
=
\langle \mathbf{II}_p(v,\cdot), Q \rangle_{\operatorname{HS}}$ for every $v \in T_pM$.
\end{lemma}

\begin{proof}
Fix $p \in M$. By Lemma~\ref{grass_tangent}, for every $(X,Q) \in (\nu_c)_p$ and $v \in T_pM$, we have
$$\langle v, X \rangle_{\mathbb{R}^D}
+ c\,\langle \mathbf{II}_p(v,\cdot), Q \rangle_{\operatorname{HS}} = 0.$$
Write $X = X^\top + X^\perp$ where $X^\top \in T_pM$ and $X^\perp \in \nu_p$. Since $\langle v, X^\perp \rangle_{\mathbb{R}^D}=0$, the condition reduces to
$\langle v, X^\top \rangle_{\mathbb{R}^D}
= -c\,\langle \mathbf{II}_p(v,\cdot), Q \rangle_{\operatorname{HS}}$ for every $v \in T_pM$. By the definition of $S_p(Q)$, this implies $X^\top = -cS_p(Q)$, which proves the desired equality. The existence and uniqueness of such a vector $S_p(Q)$ follow from the Riesz representation theorem.
\end{proof}

\begin{lemma}\label{grass_connection}
Let $M \subset \mathbb{R}^D$ be a smooth closed oriented $d$-dimensional submanifold, and let $X,Y \in \mathfrak{X}(M)$. Then
\[
\nabla^{\mathbf{Gr}_c^+(D,d)}_{d\bar{\mathbf{g}}_c^+(X)}\bigl(d\bar{\mathbf{g}}_c^+(Y)\bigr)
= (\nabla_X^M \mathbf{II})(Y,\cdot) + \mathbf{II}(\nabla^M_X Y,\cdot)
\in \operatorname{Hom}(TM,\nu).
\]
\end{lemma}

\begin{proof}
  Since a constant rescaling of a Riemannian metric does not change its Levi--Civita connection and the covering $\mathbf{Gr}^+(D,d) \to \mathbf{Gr}(D,d)$ is a local isometry, we have
  \[
  \nabla^{\mathbf{Gr}_c^+(D,d)}_{d\bar{\mathbf g}_c^+(X)}\bigl(d\bar{\mathbf g}_c^+(Y)\bigr)
  =
  \nabla^{\mathbf{Gr}^+(D,d)}_{d\bar{\mathbf g}^+(X)}\bigl(d\bar{\mathbf g}^+(Y)\bigr) = \nabla^{\mathbf{Gr}(D,d)}_{\,d\bar{\mathbf g}(X)}\bigl(d\bar{\mathbf g}(Y)\bigr).
  \]
  Applying the last formula in Lemma~\ref{grass_tangent} to $A=d\bar{\mathbf g}(Y)=\mathbf{II}(Y,\cdot)$ and using the Leibniz rule, we obtain
  \[
  \nabla^{\mathbf{Gr}(D,d)}_{d\bar{\mathbf g}(X)}\bigl(d\bar{\mathbf g}(Y)\bigr)
  =\nabla_X^M\bigl(\mathbf{II}(Y,\cdot)\bigr)
  =(\nabla_X^M\mathbf{II})(Y,\cdot)+\mathbf{II}(\nabla^{M}_X Y,\cdot),
  \]
  which proves the claim.
\end{proof}

\begin{proof}[Proof of Theorem~\ref{principal_curvature}]
Denote by $\nabla^c$ and $\nu_c$ the Levi--Civita connection on $\mathbb{R}^D \times \mathbf{Gr}^+_c(D,d)$ and the normal bundle of $(M,g_c) \subset \mathbb{R}^D \times \mathbf{Gr}_c^+(D,d)$, respectively.
For a nonzero vector $v\in T_pM$, recall that $\kappa_c(v)$ denotes the norm of the normal curvature of the graph embedding $\iota\times\bar{\mathbf g}_c^+$ at $(p,\bar{\mathbf g}_c^+(p))$ in the direction $(v,d\bar{\mathbf g}_c^+(v))$. That is,
\begin{equation}\label{kappa_c_definition}
\kappa_c(v)
=
\sup_{\substack{\eta\in(\nu_c)_p\\ \|\eta\|_2=1}}
\frac{g^{\mathbb{R}^D\times\mathbf{Gr}_c^+(D,d)}
\bigl(\mathbf{II}_{c,p}((v,d\bar{\mathbf g}_c^+(v)),(v,d\bar{\mathbf g}_c^+(v))),\eta\bigr)}
{\bigl\|(v,d\bar{\mathbf g}_c^+(v))\bigr\|_2^2}.
\end{equation}
Let $(w - cS(Q),Q)$ be an arbitrary unit vector of $\nu_c$ at $p$ as in Lemma~\ref{normal_bundle}. Then
\[
\|w\|_{2}^2 + c^2\|S(Q)\|_{2}^2 + c\|Q\|_{\operatorname{HS}}^2 = 1,
\]
where $w \perp T_pM$. By Lemma~\ref{grass_connection}, 
\[
\nabla^c_{(v,d\bar{\mathbf{g}}_c^+(v))}(v,d\bar{\mathbf{g}}_c^+(v))
= \bigl(\nabla_{v}^{\mathbb{R}^D} v,\,
\nabla^{\mathbf{Gr}_c^+(D,d)}_{d\bar{\mathbf{g}}_c^+(v)} d\bar{\mathbf{g}}_c^+(v)\bigr)
= \bigl(\nabla_{v}^M v + \mathbf{II}(v,v),\, (\nabla_{v}^M \mathbf{II})(v,\cdot) + \mathbf{II}(\nabla_{v}^M v,\cdot)\bigr).
\]
Since by Lemma~\ref{grass_tangent},
\[
d(\iota \times \bar{\mathbf{g}}_c^+)_p(\nabla_{v}^M v)
= \bigl(\nabla_{v}^M v,\, \mathbf{II}(\nabla_{v}^M v,\cdot)\bigr),
\]
the projections of $\nabla^c_{(v,d\bar{\mathbf{g}}_c^+(v))}(v,d\bar{\mathbf{g}}_c^+(v))$ and $\bigl(\mathbf{II}(v,v),(\nabla_{v}^M \mathbf{II})(v,\cdot)\bigr)$ onto $\nu_c$ coincide.
On the other hand,
\[
g^{\mathbb{R}^D \times \mathbf{Gr}_c^+(D,d)}\bigl((w,0),(\mathbf{II}(v,v),(\nabla_{v}^M \mathbf{II})(v,\cdot))\bigr)
= \bigl\langle w,\mathbf{II}(v,v)\bigr\rangle_{\mathbb{R}^D},
\]
and since $S(Q) \in T_pM$ and $\mathbf{II}(v,v)\in\nu_p$,
\[
g^{\mathbb{R}^D \times \mathbf{Gr}_c^+(D,d)}\bigl((-cS(Q),Q),(\mathbf{II}(v,v),(\nabla_{v}^M \mathbf{II})(v,\cdot))\bigr)
= c \bigl\langle Q,(\nabla_{v}^M \mathbf{II})(v,\cdot)\bigr\rangle_{\operatorname{HS}}.
\]
Therefore, 
\begin{flalign*}
  |g^{\mathbb{R}^D \times \mathbf{Gr}_c^+(D,d)}((w-cS(Q),Q), (\mathbf{II}({v},{v}) ,(\nabla_{v}^M \mathbf{II})({v},\cdot)))| &= |\langle w, \mathbf{II}(v,v) \rangle_{\mathbb{R}^D} + c\langle Q, (\nabla_{v}^M \mathbf{II})({v},\cdot) \rangle_{\operatorname{HS}}|
  \\ &\leq (\kappa(v) \|w\|_2 + c\|(\nabla_v^M \mathbf{II})(v, \cdot)\|_{\operatorname{HS}}\|Q\|_{\operatorname{HS}})
  \\ &\leq \sqrt{\kappa(v)^2 + c\|(\nabla_v^M \mathbf{II})(v, \cdot)\|_{\operatorname{HS}}^2}\sqrt{\|w\|_2^2 + c\|Q\|_{\operatorname{HS}}^2} 
  \\ &\leq \sqrt{\kappa(v)^2 + c \|(\nabla_v^M \mathbf{II})(v, \cdot)\|_{\operatorname{HS}}^2},
\end{flalign*}
where we used $\|w\|_2^2 + c^2\|S(Q)\|_2^2 + c\|Q\|_{\operatorname{HS}}^2 = 1$ in the last inequality. Since the unit vector $(w-cS(Q),Q)$ was arbitrary, taking the supremum over all unit vectors in $(\nu_c)_p$ as in \eqref{kappa_c_definition}, we obtain
$$\kappa_c(v)
\leq \frac{\sqrt{\kappa(v)^2 + c\|(\nabla_v^M\mathbf{II})(v, \cdot)\|_{\operatorname{HS}}^2}}{\|(v,d\bar{\mathbf{g}}_c^+(v))\|_2^2}
= \frac{\sqrt{\kappa(v)^2 + c\|(\nabla_v^M \mathbf{II})(v, \cdot)\|_{\operatorname{HS}}^2}}{1 + c\|\mathbf{II}(v,\cdot)\|_{\operatorname{HS}}^2}.$$
By the volume assumption in Theorem~\ref{principal_curvature}, we obtain
\begin{flalign*}
\kappa_c(v)\,(\mathrm{vol}_c(M))^{1/d}
&\leq \frac{\sqrt{\kappa(v)^2 + c\|(\nabla_v^M \mathbf{II})(v, \cdot)\|_{\operatorname{HS}}^2}}{\sqrt{1 + c\|\mathbf{II}(v,\cdot)\|_{\operatorname{HS}}^2}}\,
(\mathrm{vol}(M))^{1/d}.
\end{flalign*}
We compute
\begin{flalign*}
\frac{d}{dc}\left(\frac{\kappa(v)^2 + c\|(\nabla_v^M \mathbf{II})(v, \cdot)\|_{\operatorname{HS}}^2}{1 + c\|\mathbf{II}(v,\cdot)\|_{\operatorname{HS}}^2}\right)
&= \frac{\|(\nabla_v^M \mathbf{II})(v, \cdot)\|_{\operatorname{HS}}^2 - \kappa(v)^2\|\mathbf{II}(v,\cdot)\|_{\operatorname{HS}}^2}
{(1 + c\|\mathbf{II}(v,\cdot)\|_{\operatorname{HS}}^2)^2} < 0,
\end{flalign*}
by the assumption. This proves the claimed monotonicity. Comparing with $c=0$ proves the strict inequality.
\end{proof}

\begin{remark}
In codimension one, at each $q\in M$, let $\{e_i(q)\}_{i=1}^{D-1}$ be an orthonormal principal basis of $T_qM$ with principal curvatures
$\lambda_1(q),\dots,\lambda_{D-1}(q)$. In this basis, the metric $g_c$ satisfies
\[
  [(g_c)_q] = I + c H_q,\qquad
  H_q = \operatorname{diag}(\lambda_1(q)^2,\dots,\lambda_{D-1}(q)^2),
\]
so that $\operatorname{tr}H_q = \sum_{i=1}^{D-1} \lambda_i(q)^2$. Therefore, for small $c>0$, we obtain
\[
  d\mathrm{vol}_c
  = \sqrt{\det(I + cH_q)}\,d\mathrm{vol}
  = \Bigl(1 + \frac{c}{2}\,\sum_{i=1}^{D-1} \lambda_i(q)^2 + O(c^2)\Bigr)\,d\mathrm{vol}.
\]
In particular, regions where many principal curvatures are simultaneously large
contribute most to the first-order increase of $\mathrm{vol}_c(M)$. On the other hand, for any $p\in M$ and any unit vector $v=\sum_{i=1}^{D-1}v_i e_i(p) \in T_pM$, the volume condition in Theorem~\ref{principal_curvature} states
\[
  \mathrm{vol}_c(M) \;\leq\; \Bigl(1 + c \sum_{i=1}^{D-1}v_i^2\lambda_i(p)^2\Bigr)^{(D-1)/2} \mathrm{vol}(M) = \Bigl(1 + \frac{c\,(D-1)}{2}\sum_{i=1}^{D-1}v_i^2\lambda_i(p)^2 + O(c^2)\Bigr)\,\mathrm{vol}(M).
\]
Thus, for small $c$, one may interpret this condition as requiring that $\sum_i v_i^2\lambda_i(p)^2$ be sufficiently large relative to the average of $\frac{1}{D-1}\sum_i\lambda_i(q)^2$ over $M$. In higher codimension, the analogous comparison uses $\|\mathbf{II}_p(v,\cdot)\|_{\operatorname{HS}}^2$, which combines the squared contributions from the different normal directions.
\end{remark}

The condition
$\kappa(v) > \frac{\|(\nabla_v^M \mathbf{II})(v,\cdot) \|_{\operatorname{HS}}}{\|\mathbf{II} (v,\cdot)\|_{\operatorname{HS}}}$ in Theorem~\ref{principal_curvature}
indicates that the normal curvature of $M$ in the direction $v$ dominates the logarithmic variation of the second fundamental form in the same direction.
\begin{theorem}\label{log_II_derivative_bound}
 Let $M \subset \mathbb{R}^D$ be a smooth $d$-dimensional submanifold. For a unit-speed geodesic $\gamma \subset M$ such that $\gamma'(t)=v(t)$, suppose that the tensor $\mathbf{II}_{\gamma(t)}(v(t),\cdot)$ does not vanish along $\gamma$. Then
\[
\left|
\frac{d}{dt}\log \|\mathbf{II}(v,\cdot)\|_{\operatorname{HS}}
\right|
\le
\frac{
\bigl\|(\nabla_v^M \mathbf{II})(v,\cdot)\bigr\|_{\operatorname{HS}}
}{
\bigl\|\mathbf{II}(v,\cdot)\bigr\|_{\operatorname{HS}}
}.
\]
\end{theorem}
The proof is given in Appendix~\ref{proof_log_II}.

\begin{remark}\label{II_decrease}
The one-parameter family of metrics $\{g_c\}_{c > 0}$ may be viewed as an extrinsic
curvature-regularizing deformation of $M$ at scale $c$. Recall that $\mathbf{II}_c$ denotes the
second fundamental form of $(M,g_c) \subset \mathbb{R}^D \times
\mathbf{Gr}_c^+(D,d)$. For $p \in M$ and a unit vector $v \in T_pM$, Theorem~\ref{principal_curvature}
provides a strictly decreasing upper bound for $\kappa_c(v)(\mathrm{vol}_c(M))^{1/d}$ under its hypotheses. This suggests that the deformation
$g_c$ may reduce $\|\mathbf{II}_c\|_{2}(\mathrm{vol}_c(M))^{1/d}$ for appropriate $c>0$.
\end{remark}

\subsection{Extension of the width of bottlenecks}\label{sec:bottleneck}
Recall the definitions of $L(M)$ and $L_c'(M)$ in Definition~\ref{bottleneck_widths}. For $c>0$, define
\[
  L_c(M) = \min_{(q_1,q_2)\ \text{bottleneck}} \frac{1}{2}\,d_c(q_1,q_2).
\]
This subsection has two parts. We first estimate $L_c(M)$ and illustrate the resulting bound with the standard torus. The analysis of Euclidean bottlenecks in this part will also be used in the proof of Theorem~\ref{normalized_lifted_bottleneck}. After recalling a local comparison between intrinsic and Euclidean distances, we study pairs realizing $L_c'(M)$. By definition, we have $L_c(M) \geq L(M)$. In general, there is no meaningful lower bound for $L_c(M)$ that improves on $L(M)$ without additional assumptions on $M$. However, for closed hypersurfaces with small maximal normal curvatures, we obtain nontrivial lower bounds for $L_c(M)$ that are strictly larger than $L(M)$.

\begin{lemma}\label{antipodal}
Let $M \subset \mathbb{R}^D$ be a smooth closed connected orientable hypersurface with a unit normal vector field $\nu$. Let $q_1,q_2 \in M$ be two distinct points such that the open line segment between $q_1$ and $q_2$ is disjoint from $M$, and suppose that the vectors $\nu_{q_1}$ and $\nu_{q_2}$ are both parallel to the line through $q_1$ and $q_2$. Then the Gauss map $\bar{\mathbf{g}}^+ \colon M \to S^{D-1}$ sends $q_1$ and $q_2$ to antipodal points on $\mathbf{Gr}^+(D,D-1) \cong S^{D-1}$.
\end{lemma}

\begin{proof}
Let $\mathrm{int}(M)$ denote the bounded component of $\mathbb{R}^D\setminus M$. Since the open line segment between $q_1$ and $q_2$ does not intersect $M$, the Jordan--Brouwer separation theorem implies that it is contained entirely either in $\mathrm{int}(M)$ or in the unbounded component of $\mathbb{R}^D\setminus M$. If the segment lies inside $\mathrm{int}(M)$, then the outward normal at $q_1$ points in the direction opposite to $\overrightarrow{q_1q_2}$ while the outward normal at $q_2$ points in the direction opposite to $\overrightarrow{q_2q_1}$. If instead the segment lies outside $M$, then the outward normals at $q_1$ and $q_2$ point in the directions of $\overrightarrow{q_1q_2}$ and $\overrightarrow{q_2q_1}$, respectively. Hence $\nu_{q_1} = - \nu_{q_2}$, and the Gauss map $\bar{\mathbf{g}}^+$ sends $q_1$ and $q_2$ to antipodal points on $S^{D-1}$.
\end{proof}

\begin{theorem}\label{hypersurface}
Let $M \subset \mathbb{R}^D$ be a smooth closed connected orientable hypersurface. Suppose that the reach of $M \subset \mathbb{R}^D$ is attained at a bottleneck of $M$ (see the second case in Theorem~\ref{reach}) with $\textsf{rch}_{\mathbb{R}^D}(M) = L(M)$. Then
    $$L_{c}(M) \geq \min\!\left(\frac{1}{2}\sqrt{\,4L(M)^2 + c\pi^2\,},\, 2L(M) \right).$$
\end{theorem}

\begin{proof}
Let $\nu$ denote the unit normal vector field of $M \subset \mathbb{R}^D$ and suppose that $q_1,q_2 \in M$ form a bottleneck of $M \subset \mathbb{R}^D$. If $ \nu_{q_1}= \nu_{q_2}$ (see Figure~\ref{hypersurface_proof}), then by Lemma~\ref{antipodal}, the line segment $\overline{q_1 q_2}$ contains another point $q_3 \in M$. We will show that $\|q_1-q_3\|_{\mathbb{R}^D} \geq 2L(M)$. 

Suppose for contradiction that $\|q_1-q_3\|_{\mathbb{R}^D} < 2L(M)$ and let the midpoint $q_4 = \frac{q_1+q_3}{2}$. Then
$$ \|q_1 - q_4\|_{\mathbb{R}^D} < L(M), \quad \|q_3-q_4\|_{\mathbb{R}^D} < L(M),\quad \overline{q_1 q_4 } \perp T_{q_1}M.$$
Let $q_5\in M$ be a nearest point to $q_4$. Then
$$\|q_5-q_4\|_{\mathbb{R}^D}=d_{\mathbb{R}^D}(q_4,M)\leq\|q_3-q_4\|_{\mathbb{R}^D}=\|q_1-q_4\|_{\mathbb{R}^D}<L(M),\qquad \overline{q_5q_4}\perp T_{q_5}M.$$
By the definition of the reach, there is at most one point in $M$ satisfying the above two conditions. Hence $q_1=q_5$. Hence $d_{\mathbb{R}^D}(q_4,M)=\|q_4-q_1\|_{\mathbb{R}^D}=\|q_4-q_3\|_{\mathbb{R}^D}$, so $q_3$ is also a nearest point to $q_4$, contradicting uniqueness of the nearest point. Therefore, we conclude $\|q_1 -q_3\|_{\mathbb{R}^D} \geq 2L(M)$ and one can deduce $\|q_2 - q_3\|_{\mathbb{R}^D} \geq 2L(M)$ in the same way. Hence $\|q_1-q_2\|_{\mathbb{R}^D} = \|q_1- q_3\|_{\mathbb{R}^D} + \|q_3 - q_2\|_{\mathbb{R}^D} \geq 4L(M)$.
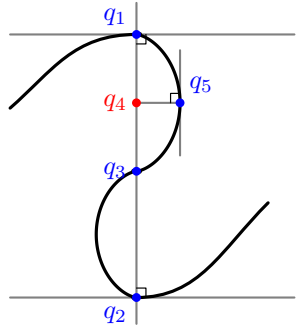
\begin{figure}[ht]
\begin{center}
\begin{tikzpicture}[line cap=round,line join=round, scale=0.7]

% ===== heights =====
\pgfmathsetmacro{\yone}{5.8}
\pgfmathsetmacro{\ythree}{3.2}
\pgfmathsetmacro{\ytwo}{0.8}

% cycloid scales (one arch fits q1->q3 and q3->q2)
\pgfmathsetmacro{\aone}{(\yone-\ythree)/(2*pi)}
\pgfmathsetmacro{\atwo}{(\ythree-\ytwo)/(2*pi)}

% key points
\coordinate (q1) at (0,\yone);
\coordinate (q3) at (0,\ythree);
\coordinate (q2) at (0,\ytwo);
\coordinate (q4) at (0,{(\yone+\ythree)/2});  % midpoint of q1 and q3 (off-curve)

% q5 at t=pi (same height as q4)
\pgfmathsetmacro{\tFive}{pi}
\pgfmathsetmacro{\qFiveX}{\aone*(1 - cos(\tFive r))} % = 2*aone
\pgfmathsetmacro{\qFiveY}{\yone - \aone*(\tFive - sin(\tFive r))}
\coordinate (q5) at (\qFiveX,\qFiveY);

% extension amount (used for cycloid arcs)
\pgfmathsetmacro{\delta}{0.6}
\pgfmathsetmacro{\SAMPLES}{500}

% ===== gray guides =====
\draw[gray,thick] (0,0.3) -- (0,6.4);                 % vertical q1--q2
\draw[gray,thick] (-2.4,\yone) -- (3.0,\yone);        % horizontal at q1
\draw[gray,thick] (-2.4,\ytwo) -- (3.0,\ytwo);        % horizontal at q2
\draw[gray,thick] (q4) -- (q5);                        % horizontal q4--q5
\draw[gray,thick] ($(q5)+(0,-1)$) -- ($(q5)+(0,1)$); % vertical at q5

% ===== main cycloid arcs (black) =====
% q1 -> q3 : right-rotated cycloid
\draw[black,very thick,domain=-\delta:2*pi+\delta,samples=\SAMPLES,variable=\t,smooth]
  plot ({ \aone*(1 - cos(\t r)) },
        { \yone - \aone*(\t - sin(\t r)) });

% q3 -> q2 : left-rotated cycloid
\draw[black,very thick,domain=-\delta:2*pi+\delta,samples=\SAMPLES,variable=\t,smooth]
  plot ({ -\atwo*(1 - cos(\t r)) },
        {  \ythree - \atwo*(\t - sin(\t r)) });

% ===== extra tails so the curve continues beyond q1 (left/down) and q2 (right/up) =====
% tail from q1: bend to the LEFT and DOWN, starting with horizontal direction
\draw[black,very thick]
  (q1) .. controls ($(q1)+(-1.2,0)$) and ($(q1)+(-1.6,-0.7)$) .. ($(q1)+(-2.4,-1.4)$);

% tail from q2: bend to the RIGHT and UP, starting with horizontal direction
\draw[black,very thick]
  (q2) .. controls ($(q2)+(1.2,0)$) and  ($(q2)+(1.6,0.9)$) .. ($(q2)+(2.5,1.8)$);

% ===== right-angle markers =====
% at q1 (horizontal vs vertical)
\draw[black,line width=0.45pt]
  ($(q1)+(0.18,0)$) -- ($(q1)+(0.18,-0.18)$) -- ($(q1)+(0,-0.18)$);
% at q2 (horizontal vs vertical)
\draw[black,line width=0.45pt]
  ($(q2)+(0.18,0)$) -- ($(q2)+(0.18,0.18)$) -- ($(q2)+(0,0.18)$);
% at q5 (vertical vs horizontal q4--q5)
\draw[black,line width=0.45pt]
  ($(q5)+(-0.18,0)$) -- ($(q5)+(-0.18,0.18)$) -- ($(q5)+(0,0.18)$);

% ===== points =====
\fill[blue] (q1) circle (2.4pt) node[above left] {$q_1$};
\fill[blue] (q3) circle (2.4pt) node[left]       {$q_3$};
\fill[blue] (q2) circle (2.4pt) node[below left] {$q_2$};
\fill[blue] (q5) circle (2.4pt) node[above right]{$q_5$};
\fill[red]  (q4) circle (2.4pt) node[left]       {$q_4$};

\end{tikzpicture}
\end{center}
\caption{Geometric configuration in the proof of Theorem~\ref{hypersurface}.}
\label{hypersurface_proof}
\end{figure}

If $\nu_{q_1} = -\nu_{q_2}$, then $\|q_1-q_2\|_{\mathbb{R}^D} \geq 2L(M)$, and the geodesic distance $d_{g^{\mathbf{Gr}^+}}(\bar{\mathbf{g}}^+(q_1),\bar{\mathbf{g}}^+(q_2))$ equals $\pi$. Hence, we have $d_c(q_1,q_2) \geq \sqrt{(2L(M))^2 + c \cdot (\pi)^2}$ for $c>0$. Considering both cases together, we conclude that
\[
 L_{c}(M) \geq \min\!\left(\frac{1}{2}\sqrt{\,4L(M)^2+ c\pi^2\,},\, 2L(M)\right).
\]
\end{proof}

\begin{theorem}\label{normalized_bottleneck}
Let $M \subset \mathbb{R}^D$ be a smooth closed connected orientable hypersurface. Suppose that $L(M) \leq \frac{1}{\|\mathbf{II}\|_2}$. Then for every $c \in \bigl(0, \tfrac{12L(M)^2}{\pi^2}\bigr]$,
\[
\frac{L_c(M)}{\mathrm{vol}_c(M)^{1/(D-1)}}
\geq
\frac{1}{2\,\mathrm{vol}(M)^{1/(D-1)}}
\sqrt{\frac{4L(M)^2 + c\pi^2}{1+c\|\mathbf{II}\|_2^2}}
>
\frac{L(M)}{\mathrm{vol}(M)^{1/(D-1)}},
\]
and the middle term in the inequalities is strictly increasing in $c$ on the interval $\bigl(0, \tfrac{12L(M)^2}{\pi^2}\bigr]$.
\end{theorem}

\begin{proof}
Let $a = \|\mathbf{II}\|_2 > 0$. Since $L(M) \leq \frac{1}{a}$, by Theorem~\ref{reach}, the reach of $M \subset \mathbb{R}^D$ is attained at a bottleneck. By Theorem~\ref{hypersurface}, for $c\in (0,\frac{12L(M)^2}{\pi^2}]$,
  $$ L_c(M) \geq \frac{1}{2}\sqrt{4L(M)^2+c\pi^2}.$$
On the other hand, by Theorem~\ref{volume},
  $$\mathrm{vol}_c(M) \leq (1+ca^2)^{(D-1)/2}\mathrm{vol}(M).$$
Combining these inequalities, we obtain
  $$\frac{L_c(M)}{\mathrm{vol}_c(M)^{1/(D-1)}} \geq \frac{1}{2(\mathrm{vol}(M))^{1/(D-1)}}\sqrt{\frac{4L(M)^2+c\pi^2}{1+ca^2}}.$$
Since $L(M) \leq \frac{1}{a}$ again, we compute
  $$\frac{d}{dc}\left(\frac{4L(M)^2+c\pi^2}{1+ca^2}\right) = \frac{\pi^2 - 4L(M)^2a^2}{(1+ca^2)^2} > 0,$$
which completes the proof.
\end{proof}

As an example, we consider the standard embedded torus \(T^2\subset\mathbb{R}^3\), for which the systole and the minimum width of bottlenecks are small compared to its volume. We show that under the embedding into \(\mathbb{R}^3\times\mathbf{Gr}_c^+(3,2)\), the normalized quantities \(\mathrm{sys}_c(T^2)^2/\mathrm{vol}_c(T^2)\) and \(L_c(T^2)^2/\mathrm{vol}_c(T^2)\) can be significantly increased. For comparison, recall Loewner's torus inequality \cite[p.~71]{pu1952some} in the classical setting, which states that every Riemannian $2$-torus $\mathbb{T}^2$ equipped with an arbitrary metric satisfies
\[
\frac{\mathrm{sys}(\mathbb{T}^2)^2}{\mathrm{vol}(\mathbb{T}^2)} \le \frac{2}{\sqrt{3}}\approx 1.1547.
\]
\begin{example}\label{torus}
Consider the standard embedded torus \(T^2\subset\mathbb{R}^3\) parametrized by
\[
T^2 \, = \,
\Bigl\{
\bigl((R+r\cos v)\cos u,\ (R+r\cos v)\sin u,\ r\sin v\bigr)
\ \Big|
u,v\in [0,2\pi),\ 0<r<R
\Bigr\}.
\]
Set $c=(R-r)^2$. As $\frac{r}{R}\to0$, the normalized quantities satisfy
\[
\frac{\mathrm{sys}_c(T^2)^2}{\mathrm{vol}_c(T^2)}
\longrightarrow
\frac{2\pi}{\displaystyle\int_0^{2\pi}\sqrt{1+\cos^2v}\,dv}\approx 0.8224,
\qquad
\frac{L_c(T^2)^2}{\mathrm{vol}_c(T^2)}
\longrightarrow
\frac{1}{\displaystyle 2\pi\int_0^{2\pi}\sqrt{1+\cos^2v}\,dv}\approx 0.02083,
\]
where \(\mathrm{sys}_c(T^2)\) and \(\mathrm{vol}_c(T^2)\) denote the systole and the Riemannian volume of \((T^2,g_c)\), respectively.
The proof is deferred to Appendix~\ref{appendix:torus-proof}.
\end{example}

\begin{theorem}[{\cite[Proposition~6.3]{niyogi2008finding}}]\label{distance}
Let $M\subset\mathbb{R}^D$ be a smooth closed connected orientable $d$-dimensional submanifold with $\mathsf{rch}_{\mathbb{R}^D}(M)=\tau>0$. If $p,q\in M$ satisfy $\|p-q\|_{\mathbb{R}^D}\leq\frac{\tau}{2}$, then
\[
d_M(p,q)\leq\tau-\tau\sqrt{1-\frac{2\|p-q\|_{\mathbb{R}^D}}{\tau}}.
\]
\end{theorem}

We prove that \(L_c'(M)\) after normalization by volume is strictly larger for sufficiently small \(c>0\) than the corresponding normalized Euclidean quantity. The quantity \(L_c'(M)\) appears in the reach bound of Theorem~\ref{length_barcodes}. The proof proceeds as follows. We first use the orthogonality condition at the endpoints of a minimizing geodesic realizing $L_c'(M)$ to show that its endpoints cannot converge to the same point as $c\downarrow0$ and that any subsequential limit is a Euclidean bottleneck. If the Euclidean distance between the limiting points is greater than $2L(M)$, the desired asymptotic estimate is immediate. In the remaining case, the estimate for Euclidean bottlenecks established above implies that the unit normal vectors at the limiting points are opposite. The Grassmannian term in $4L_c'(M)^2$ then contributes $c\pi^2$ to first order, while Federer's reach estimate shows that the possible decrease in the squared Euclidean distance is of order $c^2$. We finally combine these estimates with the volume bound in Theorem~\ref{volume}.

\begin{proof}[Proof of Theorem~\ref{normalized_lifted_bottleneck}]
The conclusion is immediate whenever $L_c'(M)=+\infty$, so we consider only values of $c$ for which $L_c'(M)$ is finite.
Let $\nu$ denote a unit normal vector field of $M\subset\mathbb{R}^D$, and set $a=\|\mathbf{II}\|_2$. By Federer's characterization of the reach \cite[Theorem~4.18]{federer1959curvature}, for all $p,q\in M$,
\begin{equation}\label{reach_chord_bound}
\bigl|\langle q-p,\nu_p\rangle_{\mathbb{R}^D}\bigr|
\leq \frac{\|p-q\|_{\mathbb{R}^D}^2}{2\tau}.
\end{equation}
Consider a pair $p,q \in M$ satisfying $\frac{1}{2}d_c(p,q) = L_c'(M)$, and let $(\gamma,\Gamma):[0,2L_c'(M)]\to\mathbb{R}^D\times\mathbf{Gr}_c^+(D,D-1)$ be a unit-speed minimizing geodesic joining $(p,\bar{\mathbf{g}}_c^+(p))$ to $(q,\bar{\mathbf{g}}_c^+(q))$. Set $\theta:=d_{\mathbf{Gr}^+(D,D-1)}\bigl(\bar{\mathbf{g}}^+(p),\bar{\mathbf{g}}^+(q)\bigr)\leq\pi$.
Since each projection of the geodesic has constant speed, we have $2L_c'(M)\gamma'(0)=q-p$ and $2L_c'(M)\|\Gamma'(0)\|_2=\theta$. For every $v \in T_pM$, the orthogonality condition at $p$ implies
\[
\langle \gamma'(0),v\rangle_{\mathbb{R}^D}
+c\langle \Gamma'(0),d\bar{\mathbf{g}}_p^+(v)\rangle_{\operatorname{HS}}=0
\]
By Lemma~\ref{grass_tangent} and the local isometry $\mathbf{Gr}^+(D,D-1)\to\mathbf{Gr}(D,D-1)$, we have $\|d\bar{\mathbf{g}}_p^+\|_2 \leq \| \mathbf{II}\|_2 =a$. Therefore,
for every $v\in T_pM$ with $\|v\|_2=1$,
\[
|\langle q-p,v\rangle_{\mathbb{R}^D}|
=2L_c'(M)|\langle\gamma'(0),v\rangle_{\mathbb{R}^D}|
\leq2cL_c'(M)\|\Gamma'(0)\|_2
\|d\bar{\mathbf g}_p^+(v)\|_{\operatorname{HS}}
\leq ca\theta.
\]
Taking the supremum over all such $v$ and applying the same argument at $q$, we obtain
\begin{equation}\label{lifted_tangent_error}
\|P_{T_pM}(q-p)\|_2\leq ca\theta,
\qquad
\|P_{T_qM}(q-p)\|_2\leq ca\theta,
\end{equation}
where $P_{T_pM}$ and $P_{T_qM}$ denote the orthogonal projections onto $T_pM$ and $T_qM$, respectively.

We next prove the estimate
\begin{equation}\label{lifted_bottleneck_asymptotic}
\liminf_{c\downarrow0}
\frac{4L_c'(M)^2-4\tau^2}{c}
\geq\pi^2.
\end{equation}
If the left-hand side is infinite, there is nothing to prove. Otherwise, choose a sequence $c_j\downarrow0$ along which the quotient converges to this liminf. We may assume that $L_{c_j}'(M)$ is finite for every $j$, and choose pairs $p_j,q_j$ realizing these minima. Set $\theta_j:=d_{\mathbf{Gr}^+(D,D-1)}\bigl(\bar{\mathbf{g}}^+(p_j),\bar{\mathbf{g}}^+(q_j)\bigr)$. After passing to a subsequence, suppose that $p_j\to p$ and $q_j\to q$.
If $p=q$, then $\|p_j-q_j\|_2\to0$, and Lemma~\ref{grass_tangent} and Theorem~\ref{distance} imply that for all sufficiently large $j$,
\[
\theta_j
\leq a\,d_M(p_j,q_j)
\leq a\tau\left(1-\sqrt{1-\frac{2\|p_j-q_j\|_2}{\tau}}\right)
\leq\frac65a\|p_j-q_j\|_2.
\]
Combining this with \eqref{lifted_tangent_error} and \eqref{reach_chord_bound}, and then dividing by $\|p_j-q_j\|_2>0$, we obtain $1\leq\frac65c_ja^2+\frac{\|p_j-q_j\|_2}{2\tau}$. The right-hand side tends to zero, which is a contradiction. Hence $p\neq q$.

Since $\theta_j\leq\pi$ and the tangent spaces vary continuously with the base point, passing to the limit in \eqref{lifted_tangent_error} shows that $q-p\perp T_pM$ and $q-p\perp T_qM$. Thus $p,q$ form a Euclidean bottleneck, so $\|p-q\|_2\geq2L(M)=2\tau$. If $\|p-q\|_2>2\tau$, then the quotient in \eqref{lifted_bottleneck_asymptotic} tends to infinity, contrary to the choice of $c_j$. Hence $\|p-q\|_2=2\tau=2L(M)$. The proof of Theorem~\ref{hypersurface} shows that a bottleneck pair with equal normal vectors has Euclidean distance at least $4L(M)$. Therefore, we have $\nu_p=-\nu_q$, so $\bar{\mathbf{g}}^+(p)$ and $\bar{\mathbf{g}}^+(q)$ are antipodal and $\theta_j\to\pi$.

It remains to control the Euclidean lengths $\|p_j-q_j\|_2$. For all sufficiently large $j$, by \eqref{lifted_tangent_error} and \eqref{reach_chord_bound},
\[
\sqrt{\|p_j-q_j\|_2^2-c_j^2a^2\pi^2}
\leq\sqrt{\|p_j-q_j\|_2^2-\|P_{T_{p_j}M}(q_j-p_j)\|_2^2}
=|\langle q_j-p_j,\nu_{p_j}\rangle_{\mathbb{R}^D}|
\leq\frac{\|p_j-q_j\|_2^2}{2\tau}.
\]
Squaring the first and last terms implies
\[
\|p_j-q_j\|_2^4-4\tau^2\|p_j-q_j\|_2^2+4\tau^2c_j^2a^2\pi^2\geq0.
\]
For all sufficiently large $j$, solving this quadratic inequality in $\|p_j-q_j\|_2^2$ implies either $\|p_j-q_j\|_2^2\leq2\tau^2-2\tau\sqrt{\tau^2-c_j^2a^2\pi^2}$ or $\|p_j-q_j\|_2^2\geq2\tau^2+2\tau\sqrt{\tau^2-c_j^2a^2\pi^2}$.
The first alternative is impossible because its right-hand side tends to zero, whereas $\|p_j-q_j\|_2^2\to4\tau^2$. Using $\sqrt{1-u}\geq1-u$ in the second alternative, we obtain
\[
\|p_j-q_j\|_2^2\geq4\tau^2-2\pi^2a^2c_j^2.
\]
Since $4L_{c_j}'(M)^2=\|p_j-q_j\|_2^2+c_j\theta_j^2$, we conclude that
\[
\liminf_{j\to\infty}
\frac{4L_{c_j}'(M)^2-4\tau^2}{c_j}
\geq
\liminf_{j\to\infty}\bigl(\theta_j^2-2\pi^2a^2c_j\bigr)
=\pi^2.
\]
This proves \eqref{lifted_bottleneck_asymptotic}.

Finally, by Theorem~\ref{volume}, $\mathrm{vol}_c(M)^{1/(D-1)}\leq\sqrt{1+ca^2}\,\mathrm{vol}(M)^{1/(D-1)}$. Since $a\tau\leq1$, we have $\pi^2>4\tau^2a^2$, and \eqref{lifted_bottleneck_asymptotic} implies $L_c'(M)>\tau\sqrt{1+ca^2}$ for all sufficiently small $c>0$. Therefore,
\[
\frac{L_c'(M)}{\mathrm{vol}_c(M)^{1/(D-1)}}
\geq\frac{L_c'(M)}{\sqrt{1+ca^2}\,\mathrm{vol}(M)^{1/(D-1)}}
>\frac{\tau}{\mathrm{vol}(M)^{1/(D-1)}}
=\frac{L(M)}{\mathrm{vol}(M)^{1/(D-1)}}.
\]
\end{proof}

\subsection{Range of radii for homotopy equivalences}\label{sec:radii}

\begin{lemma}[{\cite[Corollary~4.2]{warner1966extension}}; modified]\label{rauch}
Let $S$ be a smooth submanifold of a complete Riemannian manifold $N$. Assume that the sectional curvature of $N$ is bounded above by $K>0$, and that the operator norm of the second fundamental form of $S\subset N$ is bounded above by $\Lambda >0$. Let $t_{\mathrm{foc}}$ denote the first focal time of $S$ (see the first case in Figure~\ref{cutlocus_figure}). Then
$$t_{\mathrm{foc}}  \;\geq\; \frac{1}{\sqrt{K}} \arctan\!\left(\frac{\sqrt{K}}{\Lambda}\right).$$
\end{lemma}

\begin{theorem}\label{inj_rad}
Let $M \subset \mathbb{R}^D$ be a smooth closed orientable $d$-dimensional submanifold. Suppose that for some point $q \in M$, there exist distinct unit-speed geodesics $(\gamma_1, \Gamma_1)$ and $(\gamma_2,\Gamma_2)$ in $\mathbb{R}^D \times \mathbf{Gr}_c^+(D,d)$ satisfying
$$(\gamma_1, \Gamma_1)(0)=(\gamma_2,\Gamma_2)(0)=(q,\bar{\mathbf{g}}^+_c(q)),\qquad (\gamma_1, \Gamma_1)(T)=(\gamma_2,\Gamma_2)(T).$$
Then 
\[
T \;\geq\; \frac{\sqrt{c}\,\pi}{2}.
\]
\end{theorem}

\begin{proof}
The Euclidean space $\mathbb{R}^D$ has infinite injectivity radius. The Grassmannian $\mathbf{Gr}(D,d)$ has injectivity radius $\frac{\pi}{2}$ \cite[Section~5]{bendokat2024grassmann}. Since the oriented Grassmannian $\mathbf{Gr}^+(D,d)$ is a connected Riemannian covering of $\mathbf{Gr}(D,d)$, its injectivity radius is at least $\frac{\pi}{2}$. Under the scaled metric, the injectivity radius of $\mathbf{Gr}^+_c(D,d)$ is at least $\frac{\sqrt{c}\,\pi}{2}$. Hence, the product $\mathbb{R}^D \times \mathbf{Gr}^+_c(D,d)$ also has injectivity radius at least $\frac{\sqrt{c}\,\pi}{2}$. Consequently, if two distinct unit-speed geodesics in $\mathbb{R}^D \times \mathbf{Gr}^+_c(D,d)$ have the same endpoints and the same length, their common length is at least $\frac{\sqrt{c}\,\pi}{2}$.
\end{proof}

\begin{proof}[Proof of Theorem~\ref{length_barcodes}]
Let $\mathsf{rch}_{\mathbb{R}^D \times \mathbf{Gr}_c^+(D,d)}(M,g_c) = T$. Since the sectional curvature of the Grassmannian $(\mathbf{Gr}(D,d), g^{\mathbf{Gr}(D,d)})$ is bounded above by $2$ \cite[Proposition~4.1]{bendokat2024grassmann}, the sectional curvature of $\mathbb{R}^D \times \mathbf{Gr}^+_c(D,d)$ is bounded above by $\tfrac{2}{c}$. By Theorem~\ref{rauch}, if $T>0$ is a first focal time at some point $p \in M$, then
\[
T \;\geq\; \sqrt{\frac{c}{2}} \,\arctan \sqrt{\frac{2}{c\|\mathbf{II}_c\|_2^2}}.
\]
If there exist two distinct unit-speed distance-minimal geodesics between the same pair of points with length $T$, then by Theorem~\ref{inj_rad}, we have $T \geq \frac{\sqrt{c}\pi}{2}$. Otherwise, there exist distinct points $q_1,q_2 \in M$ and a point $(p,P) \in \mathbb{R}^D \times \mathbf{Gr}_c^+(D,d)$, together with unit-speed distance-minimal geodesics $(\gamma_1,\Gamma_1)$ and $(\gamma_2,\Gamma_2)$ from $(q_1,\bar{\mathbf{g}}_c^+(q_1))$ and $(q_2,\bar{\mathbf{g}}_c^+(q_2))$ to $(p,P)$, respectively, each of length $T$. Since $\mathbb{R}^D \times \mathbf{Gr}_c^+(D,d)$ is complete, the two geodesics meet at $(p,P)$ at angle $\pi$ by \cite[Theorem~1]{singh1987cut}. Hence, their concatenation with the second geodesic traversed in reverse, is a geodesic of length $2T$ joining $(q_1,\bar{\mathbf{g}}_c^+(q_1))$ and $(q_2,\bar{\mathbf{g}}_c^+(q_2))$. By the first variation principle, this geodesic is orthogonal to $(M,g_c)$ at both endpoints. If $T<\frac{\sqrt{c}\pi}{4}$, then it is distance-minimizing and $2T\geq 2L_c'(M)$ by the definition of $L_c'(M)$. Therefore, in the non-focal case, we obtain $T\geq\min\left\{\frac{\sqrt{c}\pi}{4},L_c'(M)\right\}$.

Combining these cases and using Theorem~\ref{cutlocus}, we obtain
\[
T \;\geq\; \min\!\left(
\sqrt{\frac{c}{2}} \,\arctan \sqrt{\frac{2}{c\|\mathbf{II}_c\|_2^2}},
\; \frac{\sqrt{c}\,\pi}{4},
\; L_c'(M)
\right).
\]
By \cite[Theorem~29]{petersen2006riemannian}, together with the bound of the sectional curvature of $\mathbb{R}^D \times \mathbf{Gr}_c^+(D,d)$ above and Theorem~\ref{inj_rad}, the convexity radius of $\mathbb{R}^D \times \mathbf{Gr}_c^+(D,d)$ is at least
\[
\min\left\{
\frac{\operatorname{inj}(\mathbb{R}^D \times \mathbf{Gr}_c^+(D,d))}{2},
\frac{\pi}{2\sqrt{2/c}}
\right\}
\geq
\min\left\{
\frac{\sqrt{c}\pi}{4},
\frac{\pi}{2\sqrt{2/c}}
\right\}
=
\frac{\sqrt{c}\pi}{4}.
\]
So every nonempty finite intersection of open balls of radius less than $\frac{\sqrt{c}\pi}{4}$ is contractible. By the nerve lemma and Theorem~\ref{grad_flow}, for every $r$ satisfying
\[
0 < r < \min\!\left(
\sqrt{\frac{c}{2}} \,\arctan \sqrt{\frac{2}{c\|\mathbf{II}_c\|_2^2}},
\; \frac{\sqrt{c}\,\pi}{4},
\; {L_c}'(M)
\right),
\]
there are homotopy equivalences
\[
\check{C}(M, \mathbb{R}^D\times\mathbf{Gr}_c^+(D,d) ;r)
\;\simeq\;
\bigcup_{q \in M} B_{\mathbb{R}^D \times \mathbf{Gr}_c^+(D,d)}\bigl((q,\bar{\mathbf{g}}_c^+(q)) , r\bigr)
\;\simeq\;
M,
\]
which completes the proof.
\end{proof}

\begin{remark}
Theorem~\ref{principal_curvature} together with Remark~\ref{II_decrease} suggests that, for suitable
choices of $c>0$, the directions with large normal curvature are damped so that
$\|\mathbf{II}_c\|_2$ can decrease relative to $\mathrm{vol}_c(M)$ under the embedding of $M$ into $\mathbb{R}^D \times \mathbf{Gr}^+_c(D,d)$. In view of
Theorem~\ref{length_barcodes}, this suggests that for such $c>0$ the first term in the lower bound for
$T$ can become larger relative to the volume, compared to the Euclidean metric. Under the hypotheses of Theorem~\ref{normalized_lifted_bottleneck}, the third term $L_c'(M)$ in the same lower bound is also larger than its Euclidean counterpart after normalization by volume, for all sufficiently small $c>0$. Taken together, these observations suggest that when \(\sqrt{c}\pi/4\) is not the smallest term in Theorem~\ref{length_barcodes}, the lower bound for the reach, and hence the length of the interval of radii for which the ambient Čech complex is guaranteed to be homotopy equivalent to \(M\), may increase after normalization by volume.
\end{remark}

\begin{remark}
Theorem~\ref{length_barcodes} provides a range of radii for which the ambient Čech complex is homotopy equivalent to $M$. For a finite subset $\mathbf{Y} \subset M$ satisfying the sampling condition in Theorem~\ref{niyogi_cutlocus}, the Čech--Rips interleaving allows the homology of $M$ to be recovered from the Vietoris--Rips persistent homology of $\mathbf{Y}$, equipped with the restriction of the ambient distance; see \cite[Remark~17]{attali2022tight}.
\end{remark}

\subsection{Stability}\label{sec:stability}
To prove Theorem~\ref{stability}, we use the following consequence of Wedin's theorem \cite[Theorem~4]{stewart1991perturbation}.

\begin{theorem}\label{wedin}
Let $A,\widetilde A \in \mathbb{R}^{D\times d}$ be full-rank matrices with $\widetilde A = A + E$.
Let $P_A$ and $P_{\widetilde A}$ denote the orthogonal projectors onto $\mathrm{Im}(A)$ and $\mathrm{Im}(\widetilde A)$, respectively.
Let $\sigma_{\min}(A)$ be the smallest singular value of $A$, and suppose that $\|E\|_{2} \le \tfrac{1}{2}\sigma_{\min}(A)$.
Then there exists a constant $C>0$, independent of $A$, $\widetilde A$, and $E$, such that
\[
  \| P_{\widetilde A} - P_A \|_{\operatorname{HS}}
  \le
  C\,\frac{\|E\|_{\operatorname{HS}}}{\sigma_{\min}(A)}.
\]
\end{theorem}

\begin{proof}[Proof of Theorem~\ref{stability}]
Let $F_i : M \to \mathbb{R}^D$ be smooth embeddings with $M_i = F_i(M) \subset \mathbb{R}^D$ for each $i \in \mathbb{N}$. Assume that there exists an embedded manifold $M_\infty = F_\infty(M) \subset \mathbb{R}^D$ such that $F_i \rightarrow F_\infty$ in the Whitney $C^1$ topology.

Using the standard coordinates on $\mathbb{R}^D$, choose finitely many coordinate charts $\{(U_\alpha,\varphi_\alpha)\}_{\alpha=1}^m$ on $M$ with connected domains and compact subsets $K_\alpha\subset U_\alpha$ whose interiors cover $M$. Since $F_\infty : M \to \mathbb{R}^D$ is an embedding and $M_\infty$ is compact, the differential
$d(F_\infty \circ \varphi_\alpha^{-1})(x)$ has full rank and depends
continuously on $x$. Hence there exists a constant $\delta > 0$ such that, for every
$\alpha$ and every $x \in \varphi_\alpha(K_\alpha)$, the smallest singular value of
$d(F_\infty \circ \varphi_\alpha^{-1})(x)$ is at least $\delta$.

Let $C>0$ be the constant in Theorem~\ref{wedin}, and choose $\varepsilon>0$ such that $\varepsilon<\min\{\delta/2,\sqrt{2}\delta/C\}$. By the Whitney $C^1$ convergence and the finiteness of the cover, there exists $n\in\mathbb{N}$ such that for every $i\geq n$ and every $\alpha$,
\[
\sup_{x \in \varphi_\alpha(K_\alpha)}
\big\|
(F_i \circ \varphi_\alpha^{-1})(x)
-(F_\infty \circ \varphi_\alpha^{-1})(x)
\big\|_{\mathbb{R}^D}
< \varepsilon,
\]
and by the equivalence of the two norms $\|\cdot\|_2$ and $\|\cdot\|_{\operatorname{HS}}$,
\[
\sup_{x \in \varphi_\alpha(K_\alpha)}
\big\|
d(F_i \circ \varphi_\alpha^{-1})(x)
-d(F_\infty \circ \varphi_\alpha^{-1})(x)
\big\|_{\operatorname{HS}}
< \varepsilon.
\]

Denote the principal angles between $\bar{\mathbf{g}}(F_i(q))$ and $\bar{\mathbf{g}}(F_\infty(q))$ by
$\theta_1(i,q) \geq \dots \geq \theta_d(i,q)$ for $i \geq n$ and $q \in M$. For a $d$–dimensional subspace $W \subset \mathbb{R}^D$ we write $P_W$ for the orthogonal
projector onto $W$. By Theorem~\ref{wedin} and the identity
$\bigl\|P_{\bar{\mathbf{g}}(F_i(q))} - P_{\bar{\mathbf{g}}(F_\infty(q))}\bigr\|_{\operatorname{HS}}^2
= 2\sum_{j=1}^d \sin^2\theta_j(i,q)$,
we obtain
\[
\sqrt{\,2\sum_{j=1}^d \sin^2\theta_j(i,q)\,} < \frac{C}{\delta}\varepsilon.
\]
Hence
\[
d_{\mathbf{Gr}(D,d)}\!\bigl(\bar{\mathbf{g}}(F_i(q)),\bar{\mathbf{g}}(F_\infty(q))\bigr)
   = \sqrt{\sum_{j=1}^d \theta_j^2(i,q)}
   \;\le\; \frac{\pi}{2}\sqrt{\sum_{j=1}^d \sin^2\theta_j(i,q)}
   \;<\; \frac{C\pi}{2\sqrt{2}\delta}\varepsilon.
\]

Suppose that there exist $x_1, x_2 \in \varphi_\alpha (U_\alpha)$ such that
$$ \det (d\big(F_i \!\circ\! \varphi_\alpha^{-1}\big)(x_1)^\top d\big(F_\infty \!\circ\! \varphi_\alpha^{-1}\big)(x_1))>0 $$
but
$$ \det (d\big(F_i \!\circ\! \varphi_\alpha^{-1}\big)(x_2)^\top d\big(F_\infty \!\circ\! \varphi_\alpha^{-1}\big)(x_2))<0.$$
Since the map
\[
x \longmapsto
\det\!\Bigl(
d(F_i \circ \varphi_\alpha^{-1})(x)^\top
d(F_\infty \circ \varphi_\alpha^{-1})(x)
\Bigr)
\]
is continuous on $\varphi_\alpha(U_\alpha)$, there exists a point
$x_3 \in \varphi_\alpha(U_\alpha)$ such that
$$ \det (d\big(F_i \!\circ\! \varphi_\alpha^{-1}\big)(x_3)^\top d\big(F_\infty \!\circ\! \varphi_\alpha^{-1}\big)(x_3))=0,$$
so that
$$d_{\mathbf{Gr}(D,d)}(\bar{\mathbf{g}}(F_i(\varphi_\alpha^{-1}(x_3))) , \bar{\mathbf{g}}(F_\infty(\varphi_\alpha^{-1}(x_3)))) \geq \frac{\pi}{2},$$
which contradicts the assumption that $\varepsilon < \tfrac{\sqrt{2} \delta}{C}$. Therefore, we may choose orientations for each $F_i(M)$ with
$$ \det (d\big(F_i \!\circ\! \varphi_\alpha^{-1}\big)(x)^\top d\big(F_\infty \!\circ\! \varphi_\alpha^{-1}\big)(x))>0 $$
for every $x \in \varphi_\alpha(U_\alpha)$ and $i \geq n$, and hence
$$d_{\mathbf{Gr}^+(D,d)}(\bar{\mathbf{g}}^+(F_i(q)) , \bar{\mathbf{g}}^+(F_\infty(q))) = d_{\mathbf{Gr}(D,d)}(\bar{\mathbf{g}}(F_i(q)) , \bar{\mathbf{g}}(F_\infty(q)))$$
for every $q \in M$. 

Therefore, for every $i \geq n$, 
$$d_H^{\mathbb{R}^D \times \mathbf{Gr}_c^+(D,d)}\left((\operatorname{id}\times\bar{\mathbf{g}}_c^+)(F_i(M)),(\operatorname{id}\times\bar{\mathbf{g}}_c^+)(F_\infty(M))\right) < \varepsilon\sqrt{1+c \frac{C^2 \pi^2}{8\delta^2}}.$$
Applying Theorem~\ref{stability_theorem} completes the proof.
\end{proof}

\section{Computational examples}\label{sec:computation}

The estimation of tangent spaces of an underlying smooth manifold via local PCA \cite{kambhatla1997dimension,abdi2010principal} has recently emerged in TDA. Besides the lifts discussed in the introduction, local PCA has been employed to approximate characteristic classes of vector bundles over the underlying manifold \cite{scoccola2023approximate}. Other approaches to estimating characteristic classes from data have been developed in \cite{tinarrage2022computing,turow2025discrete,bohlsen2025counting,paluzohidalgo2026learning}.

We outline a procedure to estimate the distance $d_c$ in Definition~\ref{grassmannian_distance} using local PCA. First, we choose a finite subset $\mathbf{X} \subset M$ of the smooth closed orientable $d$-dimensional manifold $M \subset \mathbb{R}^D$ to estimate its tangent bundle. Next, we select a subset $\mathbf{Y} \subset \mathbf{X}$ to build a Vietoris--Rips filtration. With neighborhood size chosen as $k \sim |\mathbf{X}|^{2/(d+2)}$, 
we perform local PCA at each point $x \in \mathbf{X}$ using its $k$-nearest neighbors to obtain an orthonormal basis for the estimated tangent space at $x$.
Under the sampling assumptions of Theorem~\ref{tangentspace}, the number of sample points in a sufficiently small Euclidean ball of radius $r$ is of order $|\mathbf{X}|r^d$. Hence when $\varepsilon\asymp |\mathbf{X}|^{-2/(d+2)}$, a ball of radius $\sqrt{\varepsilon}$ contains on the order of $|\mathbf{X}|^{2/(d+2)}$ sample points, which motivates the above choice of $k$.

For each $x \in M$, fix an ordered orthonormal basis for $T_xM$ compatible with the given
orientation on $M$, and let $B_x$ be the matrix whose columns form this basis. Let
$\widehat{T_xM}$ be the estimated tangent space at $x$, and choose an ordered orthonormal basis for
$\widehat{T_xM}$. Denote the corresponding basis matrix by
$\widehat{B_x}$. 
\begin{theorem}\label{tangentspace}
Let $n = |\mathbf{X}|$, and suppose that the points in $\mathbf{X}$ are sampled independently from a density $\rho\in C^3(M)$ with respect to the Riemannian volume measure, where $0<\rho_{\min}\leq\rho(x)\leq\rho_{\max}<\infty$ for all $x\in M$. Let $\varepsilon = \varepsilon(n)$ be a sequence of positive numbers with
$\varepsilon = O\bigl(n^{-\frac{2}{d+2}}\bigr)$. Assume that for each
$x_i \in \mathbf{X}$, the set $\bigl\{y \in \mathbf{X} \mid d_{\mathbb{R}^D}(x_i,y) \leq \sqrt{\varepsilon}\bigr\}$
contains at least $d+1$ points, and let $\widehat{T_{x_i}M}$
be estimated from this neighborhood by local PCA centered at $x_i$,
following \cite[Section~2]{singer2012vector} with constant weights. Then, with high
probability,
\[
d_{\mathbf{Gr}(D,d)}\bigl(T_{x_i}M,\widehat{T_{x_i}M}\bigr)
= O\bigl(\varepsilon^{\frac{3}{4}}\bigr)
= O\bigl(n^{-\frac{3}{2(d+2)}}\bigr).
\]
\end{theorem}
The proof is given in Appendix~\ref{tangentspace_proof}. To establish that such orientations can be chosen consistently across tangent spaces, we prove the following theorem.
\begin{theorem}\label{orientation}
    Let $M \subset \mathbb{R}^D$ be a smooth closed connected orientable $d$-dimensional submanifold with $\mathsf{rch}_{\mathbb{R}^D}(M) = \tau >0$. If two points $p,q \in M$ satisfy $\|p-q\|_{\mathbb{R}^D} < \frac{\tau}{2}$, then
    $$ \det (B_p^\top B_q ) > 0.$$
\end{theorem}
We postpone the proof to Appendix~\ref{orientation_proof}. If two points $x_1$ and $x_2$ in $\mathbf{X}$ are adjacent and the basis matrix $\widehat{B_{x_1}}$ is oriented, then we orient $\widehat{B_{x_2}}$ by flipping the sign of its last column whenever $\det (\widehat{B_{x_1}}^\top \widehat{B_{x_2}})<0$. This yields a consistent orientation of the estimated tangent spaces over $\mathbf{X}$. We choose \(c>0\) so that the diameter of \(\mathbf{Gr}_c^+(D,d)\) matches the estimated diameter of \(M\).
Approximating the diameter of $M$ by $\widehat{\mathrm{diam}}(M) = \mathrm{diam}(\mathbf{Y})$ and using $\mathrm{diam}(\mathbf{Gr}^+(D,d)) = \max\bigl(\pi, \frac{\pi}{2}\,\sqrt{\min(d,D-d)}\bigr)$ from \cite[Theorem 12.6]{kozlov2000geometry}, we set
\begin{equation*}\label{parameter}
c = \frac{\mathrm{diam}(\mathbf{Y})^2}{\max(\pi, \frac{\pi}{2}\sqrt{\min(d,D-d)})^2}.
\end{equation*}
We then compute the distance matrix \(\mathbf{D}\) with entries
\begin{equation*}\label{grass_distance}
\bigl(\mathbf{D}\bigr)_{ij}  = \sqrt{\|y_i - y_j\|_{\mathbb{R}^D}^2 + c \, d_{\mathbf{Gr}^+(D,d)}(\widehat{T_{y_i}M},\widehat{T_{y_j}M})^2}
\end{equation*}
for $1\leq i,j \leq |\mathbf{Y}|$ and perform Vietoris--Rips persistent homology with respect to \(\mathbf{D}\).
We compute persistent homology using the Ripser library~\cite{bauer2021ripser}, with coefficients in $\mathbb{Z}/2$.

\subsection{Time-delay embedding of the double-gyre attractor}
The driven double-gyre system in $(x,y) \in [0,2] \times [0,1]$ is specified by the stream function
\begin{equation}\label{doublegyre}
    \phi(x, y, t) = C \sin (\pi f(x,t)) \sin (\pi y),\quad f(x,t) = (\eta \sin (\omega t)) x^2 + (1-2\eta \sin (\omega t)) x.
\end{equation}
This system was first introduced in \cite{shadden2005definition} to describe geophysical flows in oceanic systems \cite{coulliette2001intergyre, poje1999geometry}.
For a continuous dynamical system $(M, \varphi)$, a compact invariant subset $A \subset M$ is called an \emph{attractor} if there exists an invariant open neighborhood $U \subset M$ of $A$ such that 
$$A = \bigcap_{t \geq 0} \{\varphi(p,t)\,:\,p\in  U\}.$$
For more details, see \cite{smale1967differentiable} or \cite{perea2019topological}. We choose $20000$ points from the system~\eqref{doublegyre} with times uniformly distributed over $[0,10000]$. The parameters are $C=0.1$, $\eta=0.1$, $\omega=\pi/5$ and the initial point is $(x_0,y_0)=(0.5,0.625)$, for which the trajectory approximates a torus.

Takens's theorem \cite{takens2006detecting} implies that, for a generic observable $h:M \to \mathbb{R}$, the topology of a smooth attractor can be recovered. For a trajectory $\{x(t)\}_{t \in [0,T]}$ in a system, the \emph{time-delay embedding} or \emph{sliding window} with delay $\tau>0$ and dimension $m \geq 1$ is
$$SW(h;\tau,m) = \bigl\{ \bigl(h(x(t)), h(x(t+\tau)), \dots, h(x(t+(m-1)\tau))\bigr) \in \mathbb{R}^m \mid t \in [0,T-(m-1)\tau] \bigr\}.$$
Following \cite{charo2020topology}, we apply this to the system~\eqref{doublegyre} with the observable $h(x,y)=x$, delay $\tau = 5$, and $m=4$, and obtain a finite subset of $\mathbb{R}^4$. The left part of Figure~\ref{doublegyre_image} illustrates the projection of this subset onto its first three coordinates.

We randomly select 1000 points from the subset and compute two persistence diagrams, one using the Euclidean distance and the other using the distance $d_c$ in Definition~\ref{grassmannian_distance}. The middle and right parts of Figure~\ref{doublegyre_image} display these diagrams. For the Euclidean case, one of the two one-dimensional homology classes is not detected, whereas the diagram computed from the distance $d_c$ exhibits both one-dimensional homology classes and one prominent two-dimensional homology class.
\begin{figure}
  \centering
  \includegraphics[scale=0.4]{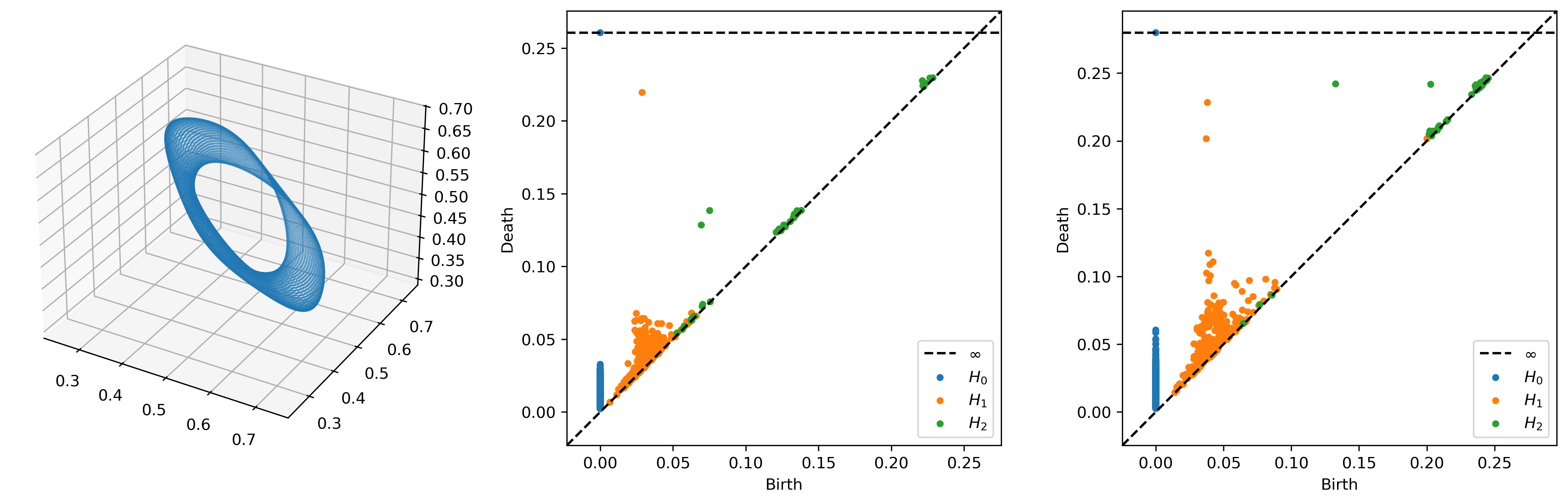}
  \caption{Projection of the time-delay embedding of the double-gyre system onto the first three coordinates (left); persistence diagrams for a subsample of time-delay embedded points with the Euclidean distance (middle) and with the distance $d_c$ (right).}
  \label{doublegyre_image}
\end{figure}

\subsection{Approximate quasi-halo orbit in the Saturn--Enceladus system}

The circular restricted three-body problem (CR3BP) models the motion of a massless particle under the gravitational attraction of two massive bodies moving on circular orbits \cite{szebehely2012theory, frauenfelder2018restricted}. In a rotating frame, the two primaries are fixed, and this system admits five equilibrium points, called \emph{Lagrange points}. Near the collinear points $L_1$ and $L_2$, there exist non-planar periodic orbits called \emph{halo orbits} \cite{connor1984three}. 

For a dynamical system $(M,\varphi)$ and the $k$-torus $\mathbb{T}^k$, a compact invariant submanifold $T \subset M$ is called a \emph{quasi-periodic torus} of the flow $\{\varphi_t\}_{t \geq 0}$, if there exist a homeomorphism 
$h : T \to \mathbb{T}^k$ and a vector $\omega \in \mathbb{R}^k$ with rationally 
independent components such that $h(\varphi_t(x)) = h(x) + t\omega$ in $\mathbb{T}^k$ for every $t \geq 0$ and $x \in T$. We perturb a halo orbit near \(L_2\) in the Saturn--Enceladus system to approximate a Liouville torus supporting a quasi-halo trajectory.

For the Saturn--Enceladus system, we work in the phase space $\mathbb{R}^6 = \{(q_1,q_2,q_3,p_1,p_2,p_3) \mid q_i,p_i \in \mathbb{R},\, 1\leq i\leq 3\}$. We approximate the quasi-halo orbit whose closure is diffeomorphic to the 3-torus $\mathbb{T}^3 = S^1 \times S^1 \times S^1$ embedded in $\mathbb{R}^6$. Applying the surface-of-section method \cite{poincare1893methodes,meiss2007differential}, we collect the transverse intersections of this orbit with the section $\{q_1 > L_2, q_2 = 0\} \subset \mathbb{R}^6$. The resulting point set approximates a \(2\)-torus \(\mathbb{T}^2=S^1\times S^1\) lying in \(\mathbb{R}^5\).

We choose 4000 points from this subset. Since the central halo orbit of the quasi-periodic trajectory is close to degeneracy, the invariant torus exhibits significant distortions. Figure~\ref{quasihalo_image} shows the degree-$0$ and degree-$1$ persistence diagrams for this subset, computed using both the Euclidean distance and the distance $d_c$ in Definition~\ref{grassmannian_distance}. Because of the narrow bottleneck and large curvature of the torus, the persistence diagram based on the Euclidean distance contains only one prominent one-dimensional homology class, whereas the diagram based on $d_c$ exhibits both.

\begin{figure}
  \centering
  \includegraphics[scale=0.4]{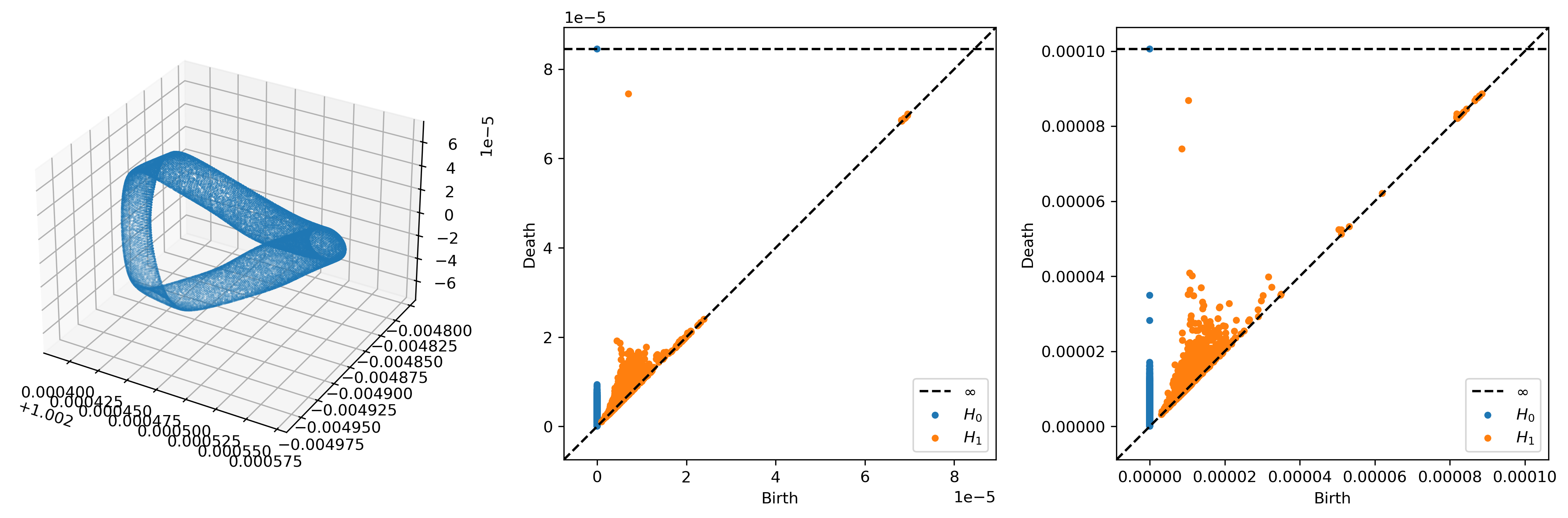}
  \caption{Three-dimensional projection of the points on the surface of section for the approximate quasi-halo orbit (left); degree-$0$ and degree-$1$ persistence diagrams for a subsample of these points, computed with the Euclidean distance (middle) and with the distance $d_c$ (right).}
  \label{quasihalo_image}
\end{figure}

\subsection{Cortical ribbon}

Persistent homology has been applied to cortical surface data, brain artery trees, functional brain networks, and neural population activity \cite{chung2009persistence,bendich2016persistent,petri2014homological,gardner2022toroidal}. We use the \texttt{fsaverage5} template \cite{fischl1999high} and the Destrieux atlas \cite{destrieux2010automatic} to extract the pial and white surfaces of the left precentral gyrus. Joining their corresponding boundary curves forms a connected closed triangulated surface with Euler characteristic $2$, representing the boundary of the cortical ribbon and having the topology of $S^2$.

\begin{figure}
  \centering
  \includegraphics[scale=0.4]{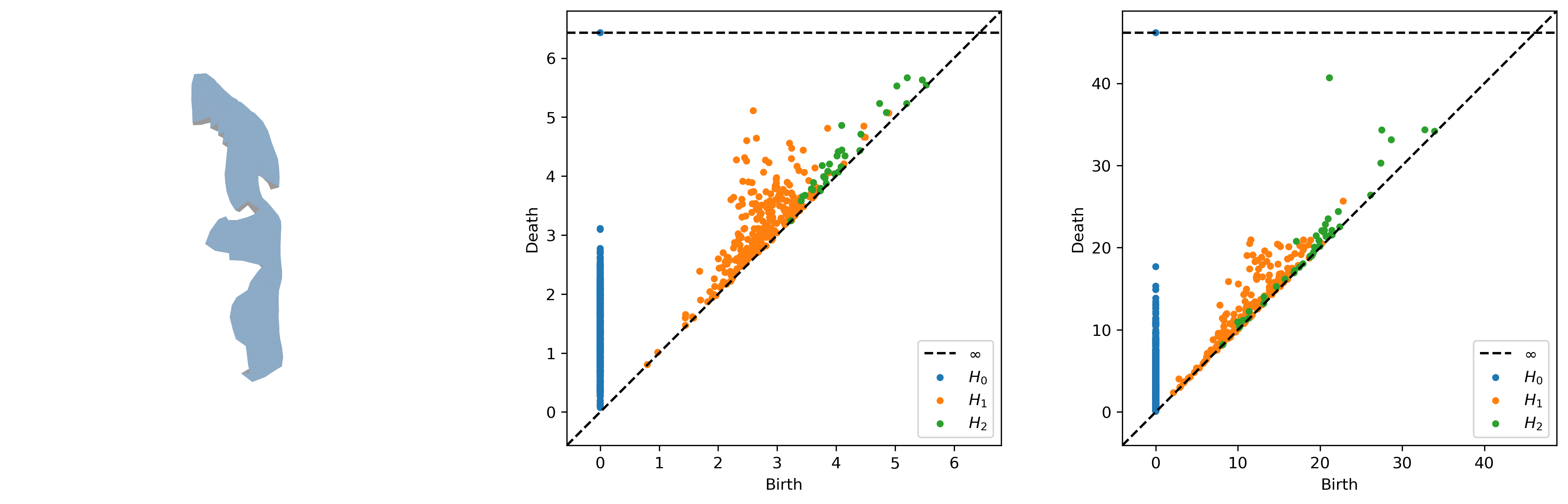}
  \caption{Boundary of the left precentral cortical ribbon (left); persistence diagrams for a sample from this surface computed with the Euclidean distance (middle) and with the distance $d_c$ (right).}
  \label{brain_surface}
\end{figure}

We sample 700 points uniformly from this surface. Since the data are given as a consistently oriented triangular mesh, we use the oriented plane of the face containing each sampled point instead of estimating tangent spaces by local PCA. The middle and right parts of Figure~\ref{brain_surface} display the resulting persistence diagrams. For this sample, the ratio of the lengths of the longest and second-longest two-dimensional bars is approximately $1.54$ for the Euclidean distance and $2.87$ for $d_c$. Thus, the longest two-dimensional bar is more clearly separated from the others under $d_c$.

\newpage
\appendix
\section{Proof of Theorem~\ref{log_II_derivative_bound}}\label{proof_log_II}
\begin{proof}[Proof of Theorem~\ref{log_II_derivative_bound}]
Note that, by the Leibniz rule,
\begin{align*}
  \nabla_v^M (\|\mathbf{II}(v,\cdot)\|_{\operatorname{HS}}^2) &= \nabla_v^M  \langle \mathbf{II}(v,\cdot), \mathbf{II}(v,\cdot) \rangle_{\operatorname{HS}}
  \\&= 2 \langle (\nabla_v^M \mathbf{II})(v,\cdot), \mathbf{II}(v,\cdot) \rangle_{\operatorname{HS}}
\end{align*}
where we used $\nabla_v^M v = 0$. On the other hand, 
$$\nabla_v^M (\|\mathbf{II}(v,\cdot)\|_{\operatorname{HS}}^2) = 2 \cdot \|\mathbf{II}(v,\cdot)\|_{\operatorname{HS}} \cdot \nabla_v^M (\|\mathbf{II}(v,\cdot)\|_{\operatorname{HS}}).$$
Therefore,
$$\nabla_v^M (\|\mathbf{II}(v,\cdot)\|_{\operatorname{HS}}) = \frac{\langle (\nabla_v^M \mathbf{II})(v,\cdot), \mathbf{II}(v,\cdot) \rangle_{\operatorname{HS}}}{\|\mathbf{II}(v,\cdot)\|_{\operatorname{HS}}}.$$
Hence,
\begin{align*}
\nabla_v^M (\log \|\mathbf{II}(v,\cdot)\|_{\operatorname{HS}}) &= \frac{\nabla_v^M (\|\mathbf{II}(v,\cdot)\|_{\operatorname{HS}})}{\|\mathbf{II}(v,\cdot)\|_{\operatorname{HS}}}
\\&= \frac{\langle (\nabla_v^M \mathbf{II})(v,\cdot), \mathbf{II}(v,\cdot) \rangle_{\operatorname{HS}}}{\|\mathbf{II}(v,\cdot)\|_{\operatorname{HS}}^2}
\end{align*}
and applying the Cauchy--Schwarz inequality proves the desired result.
\end{proof}

\section{Details of Example~\ref{torus}}\label{appendix:torus-proof}
Set $c=(R-r)^2$.
Consider the loops
\[
\gamma_1 = \{(R\cos u, R\sin u, r) : u \in [0,2\pi]\},\quad
\gamma_2 = \{(R\cos u, R\sin u, -r) : u \in [0,2\pi]\},
\]
\[
\gamma_3 = \{(R+r\cos v, 0, r\sin v) : v \in [0,2\pi]\}.
\]
Under the identification \(\pi_1(T^2)\cong\mathbb{Z}^2\), we take \([\gamma_1]=(1,0)\) and \([\gamma_3]=(0,1)\).
The shortest Euclidean representative of \([\gamma_1]\) is
\[
\gamma_4 = \{((R-r)\cos u,(R-r)\sin u,0): u\in[0,2\pi]\},
\]
whose Euclidean length is \(2\pi(R-r)\).
Hence any loop \(\gamma\subset T^2\) with \([\gamma]=(a,b)\) and \(|a|\ge 1\) satisfies
\[
\mathrm{len}_c(\gamma)\ge\mathrm{len}(\gamma)\ge 2\pi(R-r),
\]
where $\mathrm{len}_c$ and $\mathrm{len}$ denote the lengths of a loop in $(T^2,g_c)$ and in $T^2 \subset \mathbb{R}^3$, respectively. Next, we consider a loop \(\delta\subset T^2\) with \([\delta]=(0,b)\) and \(|b|\ge1\).
Such a loop meets both \(\gamma_1\) and \(\gamma_2\). Choose one intersection point with each loop. The two complementary subarcs of \(\delta\) between these points have Gauss images joining the antipodal points \(\bar{\mathbf{g}}^+(\gamma_1)\) and \(\bar{\mathbf{g}}^+(\gamma_2)\). Each image has length at least \(\pi\), so the Grassmannian component of \(\delta\) has length at least \(2\pi\).
For the product metric, this implies
\[
\mathrm{len}_c(\delta)\ge 2\pi\sqrt{c}.
\]
Combining these two cases, we obtain $\mathrm{sys}_c(T^2)\ge \min\bigl(2\pi(R-r),\,2\pi\sqrt{c}\bigr)$. With our choice \(c=(R-r)^2\), we have
\[
\mathrm{sys}_c(T^2)\ge 2\pi(R-r).
\]
In the coordinates \((u,v)\) the first and second fundamental forms of \(T^2\subset\mathbb{R}^3\) are
\[
[g] = \begin{pmatrix}
(R+r\cos v)^2 & 0 \\[2pt] 0 & r^2
\end{pmatrix},\qquad
[\mathbf{II}] = \begin{pmatrix}
-(R+r\cos v)\cos v & 0 \\[2pt] 0 & -r
\end{pmatrix},
\]
so
\[
[g_c] = \begin{pmatrix}
(R+r\cos v)^2 + c\cos^2 v & 0 \\[2pt] 0 & r^2 + c
\end{pmatrix}.
\]
It follows that
\[
\mathrm{vol}_c(T^2)
= 2\pi\sqrt{r^2+c}\int_0^{2\pi}\sqrt{(R+r\cos v)^2 + c\cos^2 v}\,dv.
\]
Let $t=r/R$. Dividing this equality by $R^2$, we obtain
\[
\frac{\mathrm{vol}_c(T^2)}{R^2}
=2\pi\sqrt{t^2+(1-t)^2}\int_0^{2\pi}\sqrt{(1+t\cos v)^2+(1-t)^2\cos^2v}\,dv,
\]
and hence
\[
\frac{\mathrm{vol}_c(T^2)}{R^2}
\longrightarrow 2\pi\int_0^{2\pi}\sqrt{1+\cos^2v}\,dv.
\]
Moreover, the matrix for $g_c$ shows that $\mathrm{len}_c(\gamma_3)=2\pi\sqrt{r^2+c}$. Together with the lower bound above, this gives
\[
2\pi R(1-t)\leq \mathrm{sys}_c(T^2)\leq 2\pi R\sqrt{t^2+(1-t)^2}.
\]
Therefore, we have $\mathrm{sys}_c(T^2)/R\to2\pi$, and
\[
\frac{\mathrm{sys}_c(T^2)^2}{\mathrm{vol}_c(T^2)}
\longrightarrow
\frac{2\pi}{\displaystyle\int_0^{2\pi}\sqrt{1+\cos^2v}\,dv}.
\]

The unit normal vector field on $T^2$ is $\nu(u,v)=(\cos u\cos v,\sin u\cos v,\sin v)$. Substituting the parametrization and this normal vector into the bottleneck conditions shows that distinct bottleneck pairs with equal normal vectors have distance $2R$, whereas those with opposite normal vectors have distance $2r$, $2(R-r)$, or $2(R+r)$. Thus, for $t\leq1/2$, the minimum distances in the two cases are $2R$ and $2r$, respectively. In the equal-normal case the Gauss images coincide, while in the opposite-normal case they are antipodal. Hence
\[
L_c(T^2)=\min\!\left(\tfrac12\sqrt{\,4r^2+c\pi^2\,},\,R\right).
\]
With \(c=(R-r)^2\) and \(t=r/R\) this equality becomes
\[
L_c(T^2)=\min\!\left(\tfrac12\sqrt{\,4R^2t^2+R^2(1-t)^2\pi^2\,},\,R\right)
\]
and we deduce $L_c(T^2)=R$ for sufficiently small $t>0$. Consequently,
\[
\frac{L_c(T^2)^2}{\mathrm{vol}_c(T^2)}
\longrightarrow
\frac{1}{\displaystyle 2\pi\int_0^{2\pi}\sqrt{1+\cos^2v}\,dv}.
\]

\section{Proof of Theorem~\ref{tangentspace}}\label{tangentspace_proof}
\begin{proof}[Proof of Theorem~\ref{tangentspace}]
Let $\theta_1,\ldots,\theta_d$ be the principal angles between $T_{x_i}M$ and $\widehat{T_{x_i}M}$. By \cite[Theorem~B.1]{singer2012vector},
\[
\min_{R\in O(d)}\|\widehat{B_{x_i}}^\top B_{x_i}-R\|_{\mathrm{HS}}=O(\varepsilon^{3/2}).
\]
Let $\widehat{B_{x_i}}^\top B_{x_i}=U\Sigma V^\top$ be a singular value decomposition, where $\Sigma=\operatorname{diag}(\cos\theta_1,\ldots,\cos\theta_d)$. For every $R\in O(d)$, we have
\[
\begin{aligned}
\|\widehat{B_{x_i}}^\top B_{x_i}-R\|_{\mathrm{HS}}^2
&=\|\Sigma-U^\top RV\|_{\mathrm{HS}}^2\\
&=\operatorname{tr}\bigl((\Sigma-U^\top RV)^\top(\Sigma-U^\top RV)\bigr)\\
&=\sum_{j=1}^d\cos^2\theta_j+d-2\operatorname{tr}(\Sigma U^\top RV)\\
&\geq\sum_{j=1}^d\cos^2\theta_j+d-2\sum_{j=1}^d\cos\theta_j\\
&=\sum_{j=1}^d(1-\cos\theta_j)^2\geq \frac{4}{\pi^4}\sum_{j=1}^d \theta_j^4
\end{aligned}
\]
for $0\leq\theta_j\leq\pi/2$. Therefore, we obtain
\[
\begin{aligned}
d_{\mathbf{Gr}(D,d)}\bigl(T_{x_i}M,\widehat{T_{x_i}M}\bigr)
&=\left(\sum_{j=1}^d\theta_j^2\right)^{1/2}\leq\left(d\sum_{j=1}^d\theta_j^4\right)^{1/4}
\\
&\leq\left(\frac{d\pi^4}{4}\right)^{1/4}
\left(\min_{R\in O(d)}\|\widehat{B_{x_i}}^\top B_{x_i}-R\|_{\mathrm{HS}}\right)^{1/2}
=O(\varepsilon^{3/4}).
\end{aligned}
\]
\end{proof}

\section{Proof of Theorem~\ref{orientation}}\label{orientation_proof}

Before proving Theorem~\ref{orientation}, we recall the following theorem. Let $M \subset \mathbb{R}^D$ be a smooth closed connected orientable $d$-dimensional submanifold, and let $\mathsf{rch}_{\mathbb{R}^D}(M)=\tau >0$.

\begin{theorem}[Proposition 6.2 in \cite{niyogi2008finding}]\label{cosine} Let $p,q \in M$, and let $\theta_1 \geq \dots \geq \theta_d \geq 0 $ be the principal angles between $T_pM$ and $T_qM$. Then 
    $$\cos \theta_1 \geq 1-\frac{1}{\tau} d_M(p,q).$$
\end{theorem}

\begin{proof}[Proof of Theorem~\ref{orientation}]
    Let $\gamma:[0,t]\to M$ be a unit-speed minimizing geodesic from $p$ to $q$. By Theorem~\ref{distance}, $t=d_M(p,q)<\tau$. Choose the matrices $B_{\gamma(s)}$ continuously along $\gamma$, with $B_{\gamma(0)}=B_p$ and with columns forming positively oriented orthonormal bases. Denote the principal angles between $T_pM$ and $T_{\gamma(s)}M$ by $\theta_1(s)\geq\dots\geq\theta_d(s)\geq0$. By Theorem~\ref{cosine}, for every $0\leq s\leq t$ and $1\leq i\leq d$,
    $$\cos\theta_i(s)\geq\cos\theta_1(s)\geq1-\frac{s}{\tau}>0.$$
    Therefore,
    $$\det (B_p^\top B_{\gamma(s)})^2=\cos^2\theta_1(s)\cdots\cos^2\theta_d(s)>0.$$
    Since the map $s\mapsto\det(B_p^\top B_{\gamma(s)})$ is continuous and $\det(B_p^\top B_{\gamma(0)})=1$, we have $\det(B_p^\top B_{\gamma(t)})>0$. Finally, replacing $B_{\gamma(t)}$ by the positively oriented basis matrix $B_q$ does not change the sign of the determinant. Hence, we conclude $\det(B_p^\top B_q)>0$.
\end{proof}

\newpage
\bibliographystyle{unsrt}
\bibliography{references}

@article{kozlov2000geometry,
  title={Geometry of Real {Grassmann} Manifolds. Part {III}},
  author={Kozlov, Sergei E},
  journal={Journal of Mathematical Sciences},
  volume={100},
  number={3},
  pages={2254--2268},
  year={2000},
  publisher={Springer},
  doi={10.1007/s10958-000-0009-1}
}

@article{balitskiy2025geometric,
  title={Geometric bounds for persistence},
  author={Balitskiy, Alexey and Coskunuzer, Baris and M{\'e}moli, Facundo},
  journal={Transactions of the American Mathematical Society},
  volume={378},
  number={12},
  pages={8437--8486},
  year={2025},
  doi={10.1090/tran/9514}
}

@article{singh1987cut,
  title={On the Cut Locus and the Focal Locus of a Submanifold in a {Riemannian} Manifold. {II}},
  author={Singh, Hukum},
  journal={Publications de l'Institut Math{\'e}matique. Nouvelle S{\'e}rie},
  volume={41(55)},
  pages={119--124},
  year={1987}
}

@article{virk20201,
  title={1-dimensional intrinsic persistence of geodesic spaces},
  author={Virk, {\v{Z}}iga},
  journal={Journal of Topology and Analysis},
  volume={12},
  number={01},
  pages={169--207},
  year={2020},
  publisher={World Scientific}
}

@article{adams2022metric,
  title={Metric thickenings and group actions},
  author={Adams, Henry and Heim, Mark and Peterson, Chris},
  journal={Journal of Topology and Analysis},
  volume={14},
  number={03},
  pages={587--613},
  year={2022},
  publisher={World Scientific}
}

@article{adams2024persistent,
  title={The persistent topology of optimal transport based metric thickenings},
  author={Adams, Henry and M{\'e}moli, Facundo and Moy, Michael and Wang, Qingsong},
  journal={Algebraic \& Geometric Topology},
  volume={24},
  number={1},
  pages={393--447},
  year={2024},
  publisher={Mathematical Sciences Publishers}
}

@book{hirsch2012differential,
  title={Differential topology},
  author={Hirsch, Morris W},
  series={Graduate Texts in Mathematics},
  volume={33},
  year={2012},
  publisher={Springer New York}
}

@book{gromov2007metric,
  title={Metric structures for Riemannian and non-Riemannian spaces},
  author={Gromov, Mikhail},
  series={Modern Birkh{\"a}user Classics},
  year={2007},
  publisher={Birkh{\"a}user Boston}
}

@incollection{stewart1991perturbation,
  title={Perturbation theory for the singular value decomposition},
  author={Stewart, Gilbert W},
  booktitle={SVD and Signal Processing II: Algorithms, Analysis and Applications},
  editor={Vaccaro, Richard J},
  pages={99--109},
  publisher={Elsevier},
  address={Amsterdam},
  year={1991}
}

@inproceedings{kim2019homotopy,
  title={Homotopy Reconstruction via the {{\v{C}}ech} Complex and the {Vietoris--Rips} Complex},
  author={Kim, Jisu and Shin, Jaehyeok and Chazal, Fr{\'e}d{\'e}ric and Rinaldo, Alessandro and Wasserman, Larry},
  booktitle={36th International Symposium on Computational Geometry ({SoCG} 2020)},
  series={Leibniz International Proceedings in Informatics ({LIPIcs})},
  volume={164},
  pages={54:1--54:19},
  publisher={Schloss Dagstuhl--Leibniz-Zentrum f{\"u}r Informatik},
  year={2020},
  doi={10.4230/LIPIcs.SoCG.2020.54}
}

@article{adams2019metric,
  title={Metric thickenings of {Euclidean} submanifolds},
  author={Adams, Henry and Mirth, Joshua},
  journal={Topology and its Applications},
  volume={254},
  pages={69--84},
  year={2019},
  publisher={Elsevier}
}

@article{basu2023connection,
  title={A connection between cut locus, {Thom} space and {Morse--Bott} functions},
  author={Basu, Somnath and Prasad, Sachchidanand},
  journal={Algebraic \& Geometric Topology},
  volume={23},
  number={9},
  pages={4185--4233},
  year={2023},
  publisher={Mathematical Sciences Publishers}
}

@article{pu1952some,
  title={Some inequalities in certain nonorientable {Riemannian} manifolds},
  author={Pu, Pao Ming},
  journal={Pacific Journal of Mathematics},
  volume={2},
  number={1},
  pages={55--71},
  year={1952},
  doi={10.2140/pjm.1952.2.55}
}

@article{federer1959curvature,
  title={Curvature measures},
  author={Federer, Herbert},
  journal={Transactions of the American Mathematical Society},
  volume={93},
  number={3},
  pages={418--491},
  year={1959},
  publisher={JSTOR}
}

@article{lim2024vietoris,
  title={{Vietoris--Rips} Persistent Homology, Injective Metric Spaces, and the Filling Radius},
  author={Lim, Sunhyuk and M{\'e}moli, Facundo and Okutan, Osman Berat},
  journal={Algebraic \& Geometric Topology},
  volume={24},
  number={2},
  pages={1019--1100},
  year={2024},
  publisher={Mathematical Sciences Publishers},
  doi={10.2140/agt.2024.24.1019}
}

@article{abdi2010principal,
  title={Principal component analysis},
  author={Abdi, Herv{\'e} and Williams, Lynne J},
  journal={WIREs Computational Statistics},
  volume={2},
  number={4},
  pages={433--459},
  year={2010},
  publisher={Wiley Online Library},
  doi={10.1002/wics.101}
}

@article{zomorodian2004computing,
  title={Computing persistent homology},
  author={Zomorodian, Afra and Carlsson, Gunnar},
  journal={Discrete \& Computational Geometry},
  volume={33},
  number={2},
  pages={249--274},
  year={2005},
  doi={10.1007/s00454-004-1146-y}
}

@book{oudot2015persistence,
  title={Persistence theory: from quiver representations to data analysis},
  author={Oudot, Steve Y},
  series={Mathematical Surveys and Monographs},
  volume={209},
  year={2015},
  publisher={American Mathematical Society},
  address={Providence, RI},
  doi={10.1090/surv/209}
}

@article{kalisnik2020finding,
  title={Finding the homology of manifolds using ellipsoids},
  author={Kali{\v{s}}nik, Sara and Le{\v{s}}nik, Davorin},
  journal={Journal of Applied and Computational Topology},
  volume={8},
  pages={193--238},
  year={2024},
  doi={10.1007/s41468-023-00145-6}
}

@article{kalivsnik2024persistent,
  title={Persistent Homology via Ellipsoids},
  author={Canova, Niklas and Kali{\v{s}}nik, Sara and Moser, Aaron and Rieck, Bastian and {\v{Z}}egarac, Ana},
  journal={arXiv preprint arXiv:2408.11450},
  year={2024},
  note={Version 3, revised 2026}
}

@article{singh1988closest,
  title={Closest point of the cut locus to submanifold},
  author={Singh, Hukum},
  journal={Proceedings of the Indian Academy of Sciences - Mathematical Sciences},
  volume={98},
  number={2--3},
  pages={179--186},
  year={1988}
}

@article{scoccola2023approximate,
  title={Approximate and discrete {Euclidean} vector bundles},
  author={Scoccola, Luis and Perea, Jose A},
  journal={Forum of Mathematics, Sigma},
  volume={11},
  pages={e20},
  year={2023},
  doi={10.1017/fms.2023.16}
}

@article{bohlsen2025counting,
  title={Counting topological interface modes using simplicial characteristic classes},
  author={Bohlsen, N and Dodin, I. Y. and Qin, H},
  journal={New Journal of Physics},
  volume={28},
  number={1},
  pages={015002},
  year={2026},
  doi={10.1088/1367-2630/ae384d}
}

@article{turow2025discrete,
  title={Discrete Approximate Circle Bundles},
  author={Turow, Brad and Perea, Jose A},
  journal={arXiv preprint arXiv:2508.12914},
  year={2025}
}

@article{tinarrage2022computing,
  title={Computing persistent {Stiefel--Whitney} classes of line bundles},
  author={Tinarrage, Rapha{\"e}l},
  journal={Journal of Applied and Computational Topology},
  volume={6},
  number={1},
  pages={65--125},
  year={2022},
  doi={10.1007/s41468-021-00080-4}
}

@article{paluzohidalgo2026learning,
  title={Learning Tangent Bundles and Characteristic Classes with Autoencoder Atlases},
  author={Paluzo-Hidalgo, Eduardo and Ike, Yuichi},
  journal={arXiv preprint arXiv:2602.22873},
  year={2026}
}

@article{tinarrage2023recovering,
  title={Recovering the homology of immersed manifolds},
  author={Tinarrage, Rapha{\"e}l},
  journal={Discrete \& Computational Geometry},
  volume={69},
  number={3},
  pages={659--744},
  year={2023},
  publisher={Springer}
}

@article{bauer2023cycling,
  title={Cycling signatures: Identifying cycling motions in time series using algebraic topology},
  author={Bauer, Ulrich and Hien, David and Junge, Oliver and Mischaikow, Konstantin},
  journal={Journal of Computational Dynamics},
  volume={12},
  number={4},
  pages={554--595},
  year={2025},
  doi={10.3934/jcd.2025007}
}

@book{do1992riemannian,
  title={Riemannian Geometry},
  author={do Carmo, Manfredo P.},
  translator={Flaherty, Francis},
  series={Mathematics: Theory and Applications},
  year={1992},
  publisher={Birkh{\"a}user Boston}
}

@book{lee2018introduction,
  title={Introduction to Riemannian Manifolds},
  author={Lee, John M},
  edition={2},
  series={Graduate Texts in Mathematics},
  volume={176},
  year={2018},
  publisher={Springer Cham},
  doi={10.1007/978-3-319-91755-9}
}

@book{petersen2006riemannian,
  author={Petersen, Peter},
  title={Riemannian geometry},
  edition={2},
  series={Graduate Texts in Mathematics},
  volume={171},
  publisher={Springer},
  address={New York},
  year={2006},
  doi={10.1007/978-0-387-29403-2}
}

@book{breiding2024metric,
  title={Metric algebraic geometry},
  author={Breiding, Paul and Kohn, Kathl{\'e}n and Sturmfels, Bernd},
  series={Oberwolfach Seminars},
  volume={53},
  year={2024},
  publisher={Birkh{\"a}user Cham},
  doi={10.1007/978-3-031-51462-3}
}

@article{aamari2019estimating,
  title={Estimating the reach of a manifold},
  author={Aamari, Eddie and Kim, Jisu and Chazal, Fr{\'e}d{\'e}ric and Michel, Bertrand and Rinaldo, Alessandro and Wasserman, Larry},
  journal={Electronic Journal of Statistics},
  volume={13},
  number={1},
  pages={1359--1399},
  year={2019},
  doi={10.1214/19-EJS1551}
}

@article{bendokat2024grassmann,
  title={A {Grassmann} manifold handbook: Basic geometry and computational aspects},
  author={Bendokat, Thomas and Zimmermann, Ralf and Absil, P-A},
  journal={Advances in Computational Mathematics},
  volume={50},
  number={1},
  pages={6},
  year={2024},
  publisher={Springer},
  doi={10.1007/s10444-023-10090-8}
}

@article{nagatomo2021holomorphic,
  title={Holomorphic maps into {Grassmann} manifolds (harmonic maps into {Grassmann} manifolds {III})},
  author={Nagatomo, Yasuyuki},
  journal={Annals of Global Analysis and Geometry},
  volume={60},
  number={1},
  pages={33--63},
  year={2021},
  doi={10.1007/s10455-021-09765-6},
  publisher={Springer}
}

@article{chazal2014persistence,
  title={Persistence stability for geometric complexes},
  author={Chazal, Fr{\'e}d{\'e}ric and de Silva, Vin and Oudot, Steve},
  journal={Geometriae Dedicata},
  volume={173},
  number={1},
  pages={193--214},
  year={2014},
  publisher={Springer}
}

@article{singer2012vector,
  title={Vector diffusion maps and the connection {Laplacian}},
  author={Singer, Amit and Wu, H-T},
  journal={Communications on Pure and Applied Mathematics},
  volume={65},
  number={8},
  pages={1067--1144},
  year={2012},
  publisher={Wiley Online Library},
  doi={10.1002/cpa.21395}
}

@inproceedings{attali2022tight,
  title={Tight Bounds for the Learning of Homotopy {\`a} la {Niyogi}, {Smale}, and {Weinberger} for Subsets of {Euclidean} Spaces and of {Riemannian} Manifolds},
  author={Attali, Dominique and Dal Poz Kou{\v{r}}imsk{\'a}, Hana and Fillmore, Christopher and Ghosh, Ishika and Lieutier, Andr{\'e} and Stephenson, Elizabeth and Wintraecken, Mathijs},
  booktitle={40th International Symposium on Computational Geometry ({SoCG} 2024)},
  series={Leibniz International Proceedings in Informatics ({LIPIcs})},
  volume={293},
  pages={11:1--11:19},
  publisher={Schloss Dagstuhl--Leibniz-Zentrum f{\"u}r Informatik},
  year={2024},
  doi={10.4230/LIPIcs.SoCG.2024.11}
}

@article{niyogi2008finding,
  title={Finding the homology of submanifolds with high confidence from random samples},
  author={Niyogi, Partha and Smale, Stephen and Weinberger, Shmuel},
  journal={Discrete \& Computational Geometry},
  volume={39},
  number={1--3},
  pages={419--441},
  year={2008},
  publisher={Springer},
  doi={10.1007/s00454-008-9053-2}
}

@article{prasad2023cut,
  title={Cut locus of submanifolds: A geometric and topological viewpoint},
  author={Prasad, Sachchidanand},
  journal={arXiv preprint arXiv:2303.14931},
  year={2023}
}

@article{shadden2005definition,
  title={Definition and properties of {Lagrangian} coherent structures from finite-time {Lyapunov} exponents in two-dimensional aperiodic flows},
  author={Shadden, Shawn C and Lekien, Francois and Marsden, Jerrold E},
  journal={Physica D: Nonlinear Phenomena},
  volume={212},
  number={3-4},
  pages={271--304},
  year={2005},
  publisher={Elsevier}
}

@article{coulliette2001intergyre,
  title={Intergyre transport in a wind-driven, quasigeostrophic double gyre: An application of lobe dynamics},
  author={Coulliette, Chad and Wiggins, Stephen},
  journal={Nonlinear Processes in Geophysics},
  volume={8},
  number={1/2},
  pages={69--94},
  year={2001},
  publisher={Copernicus Publications G{\"o}ttingen, Germany}
}

@article{poje1999geometry,
  title={Geometry of cross-stream mixing in a double-gyre ocean model},
  author={Poje, AC and Haller, G},
  journal={Journal of Physical Oceanography},
  volume={29},
  number={8},
  pages={1649--1665},
  year={1999}
}

@article{charo2020topology,
  title={Topology of Dynamical Reconstructions from {Lagrangian} Data},
  author={Char{\'o}, Gisela D and Artana, Guillermo and Sciamarella, Denisse},
  journal={Physica D: Nonlinear Phenomena},
  volume={405},
  pages={132371},
  year={2020},
  publisher={Elsevier},
  doi={10.1016/j.physd.2020.132371}
}

@article{smale1967differentiable,
  title={Differentiable dynamical systems},
  author={Smale, Stephen},
  journal={Bulletin of the American Mathematical Society},
  volume={73},
  number={6},
  pages={747--817},
  year={1967}
}

@article{perea2019topological,
  title={Topological time series analysis},
  author={Perea, Jose A},
  journal={Notices of the American Mathematical Society},
  volume={66},
  number={5},
  pages={686--694},
  year={2019},
  publisher={American Mathematical Society, AMS}
}

@incollection{takens2006detecting,
  title={Detecting strange attractors in turbulence},
  author={Takens, Floris},
  booktitle={Dynamical Systems and Turbulence, Warwick 1980},
  editor={Rand, David A. and Young, Lai-Sang},
  series={Lecture Notes in Mathematics},
  volume={898},
  pages={366--381},
  year={1981},
  publisher={Springer Berlin Heidelberg},
  doi={10.1007/BFb0091924}
}

@article{warner1966extension,
  title={Extensions of the {Rauch} Comparison Theorem to Submanifolds},
  author={Warner, F. W.},
  journal={Transactions of the American Mathematical Society},
  volume={122},
  number={2},
  pages={341--356},
  year={1966},
  doi={10.1090/S0002-9947-1966-0200873-6}
}

@article{lai2025simple,
  title={Simple Matrix Expressions for the Curvatures of {Grassmannian}},
  author={Lai, Zehua and Lim, Lek-Heng and Ye, Ke},
  journal={Foundations of Computational Mathematics},
  volume={26},
  number={4},
  pages={2053--2089},
  year={2026},
  publisher={Springer},
  doi={10.1007/s10208-025-09723-9}
}

@book{szebehely2012theory,
  title={Theory of Orbits: The Restricted Problem of Three Bodies},
  author={Szebehely, Victor G.},
  year={1967},
  publisher={Academic Press},
  address={New York}
}

@book{frauenfelder2018restricted,
  title={The restricted three-body problem and holomorphic curves},
  author={Frauenfelder, Urs and van Koert, Otto},
  series={Pathways in Mathematics},
  year={2018},
  publisher={Birkh{\"a}user Cham},
  doi={10.1007/978-3-319-72278-8}
}

@article{connor1984three,
  title={Three-Dimensional, Periodic, `Halo' Orbits},
  author={Howell, Kathleen C.},
  journal={Celestial Mechanics},
  volume={32},
  number={1},
  pages={53--71},
  year={1984},
  publisher={Springer},
  doi={10.1007/BF01358403}
}

@article{bauer2021ripser,
  title={Ripser: efficient computation of {Vietoris--Rips} persistence barcodes},
  author={Bauer, Ulrich},
  journal={Journal of Applied and Computational Topology},
  volume={5},
  number={3},
  pages={391--423},
  year={2021},
  publisher={Springer}
}

@article{chazal2021introduction,
  title={An introduction to topological data analysis: fundamental and practical aspects for data scientists},
  author={Chazal, Fr{\'e}d{\'e}ric and Michel, Bertrand},
  journal={Frontiers in Artificial Intelligence},
  volume={4},
  pages={667963},
  year={2021},
  publisher={Frontiers Media SA}
}

@article{kambhatla1997dimension,
  title={Dimension Reduction by Local Principal Component Analysis},
  author={Kambhatla, Nandakishore and Leen, Todd K},
  journal={Neural Computation},
  volume={9},
  number={7},
  pages={1493--1516},
  year={1997},
  doi={10.1162/neco.1997.9.7.1493}
}

@book{poincare1893methodes,
  title={Les m{\'e}thodes nouvelles de la m{\'e}canique c{\'e}leste},
  author={Poincar{\'e}, Henri},
  volume={2},
  year={1893},
  publisher={Gauthier-Villars et fils, imprimeurs-libraires}
}

@book{meiss2007differential,
  title={Differential dynamical systems},
  author={Meiss, James D},
  year={2007},
  publisher={SIAM}
}

@inproceedings{chung2009persistence,
  title={Persistence diagrams of cortical surface data},
  author={Chung, Moo K. and Bubenik, Peter and Kim, Peter T.},
  booktitle={Information Processing in Medical Imaging},
  series={Lecture Notes in Computer Science},
  volume={5636},
  pages={386--397},
  year={2009},
  publisher={Springer},
  doi={10.1007/978-3-642-02498-6_32}
}

@article{bendich2016persistent,
  title={Persistent homology analysis of brain artery trees},
  author={Bendich, Paul and Marron, J. S. and Miller, Ezra and Pieloch, Alex and Skwerer, Sean},
  journal={The Annals of Applied Statistics},
  volume={10},
  number={1},
  pages={198--218},
  year={2016},
  doi={10.1214/15-AOAS886}
}

@article{petri2014homological,
  title={Homological scaffolds of brain functional networks},
  author={Petri, Giovanni and Expert, Paul and Turkheimer, Federico and Carhart-Harris, Robin and Nutt, David and Hellyer, Peter J. and Vaccarino, Francesco},
  journal={Journal of the Royal Society Interface},
  volume={11},
  number={101},
  pages={20140873},
  year={2014},
  doi={10.1098/rsif.2014.0873}
}

@article{gardner2022toroidal,
  title={Toroidal topology of population activity in grid cells},
  author={Gardner, Richard J. and Hermansen, Erik and Pachitariu, Marius and Burak, Yoram and Baas, Nils A. and Dunn, Benjamin A. and Moser, May-Britt and Moser, Edvard I.},
  journal={Nature},
  volume={602},
  pages={123--128},
  year={2022},
  doi={10.1038/s41586-021-04268-7}
}

@article{fischl1999high,
  title={High-resolution intersubject averaging and a coordinate system for the cortical surface},
  author={Fischl, Bruce and Sereno, Martin I. and Tootell, Roger B. H. and Dale, Anders M.},
  journal={Human Brain Mapping},
  volume={8},
  number={4},
  pages={272--284},
  year={1999}
}

@article{destrieux2010automatic,
  title={Automatic parcellation of human cortical gyri and sulci using standard anatomical nomenclature},
  author={Destrieux, Christophe and Fischl, Bruce and Dale, Anders and Halgren, Eric},
  journal={NeuroImage},
  volume={53},
  number={1},
  pages={1--15},
  year={2010},
  doi={10.1016/j.neuroimage.2010.06.010}
}

\bigskip
\noindent\textbf{Author Information.}

\noindent\textsc{Dongwoo Gang} \\
Department of Mathematical Sciences, Seoul National University \\
\textit{Email:} dongwoo.gang@snu.ac.kr

\end{document}